\documentclass[reqno, 11pt]{amsart}
\usepackage{amssymb, amsmath, amsthm, mathrsfs, amsfonts, latexsym}
\usepackage[margin=1in]{geometry}

\newtheorem{theorem}{Theorem}[section]
\newtheorem{lemma}[theorem]{Lemma}

\newtheorem{proposition}[theorem]{Proposition}

\newtheorem{remark}[theorem]{Remark}

\theoremstyle{definition}

\newtheorem{assumption}[theorem]{Assumption}

\numberwithin{equation}{section}

\def\R{{\mathbb R}}

\def\C{{\mathbb C}}
\def\eps{\varepsilon}
\def\Re {{\rm Re}\,}
\def\Im {{\rm Im}\,}
\def\E{{\mathbb E}}
\def\P{{\mathbb P}}
\def\Z{{\mathbb Z}}

\begin{document}

\title[Hausdorff measure of the range and level sets of Gaussian random fields]
{The Hausdorff measure of the range and level sets of
Gaussian random fields with sectorial local nondeterminism}
\author{Cheuk Yin Lee}
\address{Institut de math\'ematiques, \'Ecole polytechique f\'ed\'erale de Lausanne,
Station 8, CH-1015 Lausanne, Switzerland}
\email{cheuk.lee@epfl.ch}



\begin{abstract}
We determine the exact Hausdorff measure functions for the range and level sets 
of a class of Gaussian random fields satisfying sectorial local nondeterminism and 
other assumptions. We also establish a Chung-type law of the iterated logarithm.
The results can be applied to the Brownian sheet,
fractional Brownian sheets whose Hurst indices are the same in all directions, and 
systems of linear stochastic wave equations in one spatial dimension driven by 
space-time white noise or colored noise.
\end{abstract}

\maketitle

\section{Introduction}

Consider a centered, continuous, $\R^d$-valued Gaussian random field 
$v = \{ v(x) : x \in \R^N\}$ with i.i.d.\ components, that is, $v(x) = (v_1(x), \dots, v_d(x))$ 
and $v_1, \dots, v_d$ are i.i.d.

The class of Gaussian random fields $v$ that we treat in this paper
includes the Brownian sheet and other Gaussian random fields that share a similar 
structure with the Brownian sheet, such as the fractional Brownian sheets and the 
solution of a system of linear stochastic wave equations.
One of the important properties that they have in common is that they all satisfy sectorial 
local nondeterminism (sectorial LND).
The purpose of this paper is to establish a unified and general framework that 
incorporates these Gaussian random fields and allows us to study fine properties such 
as the Hausdorff measures of the associated random sets and sharp regularity 
properties for these random fields. 
In the main results of this paper, we identify the exact Hausdorff measure functions for 
the range and level sets of $v$. 
We also prove a Chung-type law of the iterated logarithm (LIL).
We use these results to outline subtle differences between Gaussian random fields with 
sectorial LND and those with stationary increments and strong LND.

It is well known that the fractional Brownian motion satisfies strong LND \cite{P78}.
The exact Hausdorff measure function for the range of
L\'evy's multiparameter Brownian motion was identified by Goldman \cite{G88}
and the result was extended by Talagrand \cite{T95} to the case of the 
fractional Brownian motion.
For Gaussian random fields with stationary increments and 
strong LND, the problem of finding the Hausdorff 
measure functions for the range and level sets was studied by Xiao \cite{X96, X97}. 
Baraka and Mountford \cite{BM11} improved the result of Xiao \cite{X97} 
in the case of fractional Brownian motion and determined
the exact value of the Hausdorff measure of the zero set.
The Hausdorff measure function for the range of anisotropic Gaussian random fields
was studied by Luan and Xiao \cite{LX12}.

On the other hand, the Brownian sheet does not satisfy strong LND (see \cite{AWX08}).
Instead, it satisfies a weaker condition called sectorial LND \cite{KWX06, KX07}.
The Hausdorff measure functions for the range and graph of the Brownian sheet
were determined by Ehm \cite{E81}.
The Brownian sheet is closely related to the additive Brownian motion, and
the exact Hausdorff measure of the zero set of the latter was studied by 
Mountford and Nualart \cite{MN04}.
In \cite{WX11}, Wu and Xiao studied the local times of anisotropic Gaussian random 
fields with sectorial LND and obtained a partial result for the Hausdorff measure function 
for the level sets.
We point out that the exponent of the logarithmic factor in the Hausdorff measure 
function for Gaussian fields with sectorial LND is 
typically not the same as those with stationary increments and strong LND,
indicating subtle differences between these two classes of Gaussian fields.

Our framework is inspired by the special structure of the Brownian sheet.
Let us recall the local representation and behaviour of the Brownian sheet, 
which has been discussed in the literature, e.g.\ \cite{KS99, D03}.
Let $B = \{B(t_1, t_2) : (t_1, t_2) \in \R_+^2 \}$ be a Brownian sheet on $\R_+^2$.
Fix $s = (s_1, s_2) \in (0, \infty)^2$ and let $I = [s_1, s_1 + r] \times [s_2, s_2 + r]$ 
be a small interval. Then for all $t = (t_1, t_2) \in I$, we can write
\begin{equation}\label{Bs1}
B(t_1, t_2) = B(s_1, s_2) + \tilde B^1(t_1) + \tilde B^2(t_2) + R(t_1, t_2),
\end{equation}
where 
\begin{align*}
&\tilde B^1(t_1) = B(t_1, s_2) - B(s_1, s_2), \qquad 
\tilde B^2(t_2) = B(s_1, t_2) - B(s_1, s_2)\\
&\text{and}\qquad R(t_1, t_2) = B(t_1, t_2) - B(t_1, s_2) - B(s_1, t_2) + B(s_1, s_2).
\end{align*}
Note that $\{\tilde B^1(t_1) : t_1 \ge s_1 \}$ and $\{ \tilde B^2(t_2) : t_2 \ge s_2 \}$
are two independent Brownian motions (up to scaling).
Since $\mathrm{Var}(R(t_1, t_2)) = (t_1 - s_1)(t_2 - s_2)$, the remainder process $R$
is of smaller order compared to $\tilde B^1$ and $\tilde B^2$.
Now consider the increments of $B$ over the interval $I$: for all $t, t' \in I$,
\begin{equation}\label{Bs2}
B(t) - B(t') = [\tilde B^1(t_1) - \tilde B^1(t_1')] + 
[\tilde B^2(t_2) - \tilde B^2(t_2')] + \text{remainder}.
\end{equation}
This suggests that the increments of $B$ can be approximated by those of 
$\tilde B^1$ and $\tilde B^2$.

In this paper, we will consider the examples of fractional Brownian sheets and the 
solution of a linear stochastic wave equation (driven by a possibly colored noise),
which share similar structures with the Brownian sheet. 
A technical difference for these examples of Gaussian fields is that the 
corresponding processes $\tilde B^j$'s may not have 
independent increments or Markov property unlike the Brownian sheet case, where
the $\tilde B^j$'s are Brownian motions.
Nonetheless, all of these Gaussian fields have certain harmonizable representation, 
which is a useful representation to create independence (see \cite{DMX17, DLMX}).
These observations motivate us to introduce the processes $\tilde v^j$, 
$j = 1, \dots, N$, and the conditions for these processes in Assumption \ref{a2} below.

Let $0 < \alpha < 1$ be a constant and $T$ be a compact interval in $\R^N$.
Define ${\|X\|}_{L^2} = (\E|X|^2)^{1/2}$ for a random vector $X$, where 
$|\cdot|$ denotes the Euclidean norm.
We will consider the following assumptions for the Gaussian random field $v$.
The first assumption specifies upper and lower bounds for the increments of $v$ 
in $L^2$-norm, and the property of sectorial local nondeterminism.
This assumption appears in \cite{X, WX11}.

\begin{assumption}\label{a1}
There exist positive finite constants $c_0, c_1$ such that the following
properties hold:
\begin{enumerate}
\item[(a)] For all $x, y \in T$,
\[ c_0^{-1} |x-y|^\alpha \le {\|v(x) - v(y)\|}_{L^2} \le c_0|x-y|^\alpha. \]
\item[(b)] (Sectorial local nondeterminism) For all integers $n \ge 1$, 
for all $x, y^1, \dots, y^n \in T$,
\[ \mathrm{Var}(v_1(x)|v_1(y^1), \dots, v_1(y^n)) \ge 
c_1 \sum_{j=1}^N \min_{0 \le i \le n} |x_j - y^i_j|^{2\alpha}, \]
where $x = (x_1, \dots, x_N)$ and $y^0 = 0$.
\end{enumerate}
\end{assumption}

\smallskip

For any $s \in \R^N$ and $r > 0$, let 
$I_r(s)$ denote the compact interval $T \cap \prod_{j=1}^N [s_j - r, s_j + r]$.
Let $\mathscr{B}(\R_+)$ denote the $\sigma$-algebra of Borel sets in $\R_+ = [0, \infty)$.
The following assumption is a modification of Assumption 2.1 in \cite{DMX17} 
(see also \cite{DLMX}).
Here, this assumption is adapted to the structure of the Brownian sheet 
as discussed above so that at the same time we are able to deal with other Gaussian 
random fields sharing a similar structure.

\begin{assumption}\label{a2}
For every $s \in T$, there exists $0 < r_0 \le 1$, there exists, 
for each $j= 1, \dots, N$, an $\R^d$-valued Gaussian process 
$\{ \tilde v^j(A, x_j) : A \in \mathscr{B}(\R_+), x_j \in [s_j - r_0, s_j + r_0] \}$,
which has i.i.d.\ components and is a.s.~continuous in $x_j$ for fixed $A$, and
there exists a finite constant $c_2$ (depending on $T$ but not on $s$) such that 
the following properties hold:
\begin{enumerate}
\item[(a)] 
For each $j$ and $x_j$, $A \mapsto \tilde v^j(A, x_j)$ is an $\R^d$-valued 
independently scattered Gaussian measure.
Moreover, whenever $A$ and $B$ are fixed, disjoint sets,
the $\sigma$-algebra $\sigma\{\tilde v^1(A, \cdot), \dots, \tilde v^N(A, \cdot)\}$
is independent of $\sigma\{\tilde v^1(B, \cdot), \dots, \tilde v^N(B, \cdot) \}$.
\item[(b)] For all $j$, for all $0 < r \le r_0$ and all $x_j, y_j \in [s_j - r, s_j+r]$,
\begin{equation*}
{\| \tilde v^j(\R_+, x_j) - \tilde v^j(\R_+, y_j) \|}_{L^2} 
\le c_2 |x_j - y_j|^\alpha.
\end{equation*}
\item[(c)] Let $\tilde v(A, x) = \sum_{j=1}^N \tilde v^j(A, x_j)$.
The process $v(\cdot) - \tilde v(A, \cdot)$ has i.i.d.\ components and
there exist constants $0 \le \gamma_1 < \alpha$, $\gamma_2 > 0$ and $a_0 \ge 0$ 
such that for all $0 < r \le r_0$, for all $x, y \in I_r(s)$, for all $a_0 \le a < b \le \infty$,
\begin{equation*}
\begin{split}
&{\|v(x) - v(y) - \tilde v([a, b), x) + \tilde v([a, b), y)\|}_{L^2}\\
&\qquad \le c_2 \left(a^{1-\alpha} |x-y| + r^{\gamma_1}b^{\gamma_1-\alpha} 
+ r^{\gamma_2} |x-y|^{\alpha} \right).
\end{split}
\end{equation*}
\end{enumerate}
\end{assumption}

\smallskip

The next assumption originates from Assumption 2.4 in \cite{DMX17}.
Part (a) states a uniform lower bound for the $L^2$-norm of $v$, and 
part (b) basically states that for any $x$, one can find a reference point $x'$ 
(which may depend on $x$) such that for all $y, \bar y$ in a small neighbourhood of $x$,
the covariances in \eqref{eq:a3} are smoother than what one gets from the Cauchy--Schwarz inequality and the upper bound in Assumption \ref{a1}(a) above.

\begin{assumption}\label{a3}
There exists constants $0 < \eps_0 \le 1$, $c_3 > 0$, 
$c > 0$ and $\delta \in (\alpha, 1]$ such that 
the following properties hold:
\begin{enumerate}
\item[(a)] For all $x \in T^{(\eps_0)}$, ${\|v(x)\|}_{L^2} \ge c_3$, 
where $T^{(\eps_0)}$ denotes the $\eps_0$-neighbourhood
of $T$ in the Euclidean norm.
\item[(b)] For every sufficiently small compact interval $I \subset T$, 
and every $0 < \rho \le \eps_0$, there exists a finite constant $c_4$ such that 
for any $x \in I$, there exists $x' \in I^{(c\rho)}$ such that for all $y, \bar y \in I^{(c\rho)}$
with $|x - y| \le 2 \rho$ and $|x - \bar y| \le 2 \rho$, for all $i = 1, \dots, d$,
\begin{equation}\label{eq:a3}
|\E[(v_i(y) - v_i(\bar y))v_i(x')]| \le c_4 |y - \bar y|^\delta.
\end{equation}
\end{enumerate}
\end{assumption}

\medskip

Let $\mathscr{I}(T)$ be the collection of all compact intervals in $T$, and 
let $\lambda$ be the Lebesgue measure on $\R^N$.
For any subset $F$ of $\R^N$ or $\R^d$, 
let $\mathscr{H}_\phi(F)$ denote the Hausdorff measure of $F$
with respect to the function $\phi$, and
for any interval $J \subset \R^N$, let $L(\cdot, J)$ denote the local time of $v$ on $J$
(see Section 2 for definitions). 
The range of $v$ on $J$ is the random set in $\R^d$ defined by
\[ v(J) = \{ v(x) : x \in J \}. \]
For any $z \in \R^d$, the $z$-level set of $v$ on $J$ is the random set in $\R^N$ 
defined by
\[ v^{-1}(z) \cap J = \{ x \in J : v(x) = z \}. \]

The following are the main results of the paper, concerning the Hausdorff measure
of the range and level sets of $v$.

\begin{theorem}\label{thm1}
Let $T$ be a compact interval in $\R^N$.
Under Assumptions \ref{a1} and \ref{a2}, if $N < \alpha d$, then
there exist positive finite constants $C_1$ and $C_2$ such that 
\begin{equation}\label{thm1eq0}
\P\Big\{ C_1 \lambda(J) \le \mathscr{H}_\phi(v(J)) \le C_2 \lambda(J) \text{ for all }
J \in \mathscr{I}(T) \Big\} = 1,
\end{equation}
where $\phi(r) = r^{N/\alpha}(\log\log(1/r))^N$.
\end{theorem}

We can compare Theorem \ref{thm1} to the case of 
a fractional Brownian motion $X$ from $\R^N$ to $\R^d$, 
for which we have
$0 < \mathscr{H}_\psi(X(J)) < \infty$, where $\psi(r) = r^{N/\alpha}\log\log(1/r)$;
see Talagrand \cite{T95}.

Also, we remark that by Corollary 5.1 of \cite{P78}, 
if $N > \alpha d$ and $T$ has interior points, 
then a.s.\ $v(T)$ has positive Lebesgue measure in $\R^d$.
For the critical case $N = \alpha d$, 
the problem of determining the exact Hausdorff measure
function for $v(T)$ is open; 
see Talagrand \cite{T98} for some results for the fractional Brownian motion.

\begin{theorem}\label{thm2}
Let $T$ be a compact interval in $\R^N$.
Under Assumptions \ref{a1}, \ref{a2} and \ref{a3}, if $N > \alpha d$, then
there exists a positive constant $C$ such that for any fixed $z \in \R^d$,
\begin{equation}\label{thm2eq0}
\P\Big\{C L(z, J) \le \mathscr{H}_\varphi(v^{-1}(z) \cap J) < \infty \text{ for all }
J \in \mathscr{I}(T) \Big\} = 1,
\end{equation}
where $\varphi(r) = r^{N - \alpha d}(\log\log(1/r))^{\alpha d}$.
Moreover, if $N \le \alpha d$, then $v^{-1}(z) \cap T = \varnothing$ a.s.
\end{theorem}

Baraka and Mountford \cite{BM11} proved that if $N > \alpha d$
and $X$ is a fractional Brownian motion from $\R^N$ to $\R^d$, 
then $\mathscr{H}_\psi(X^{-1}(0) \cap J) = C L(0, J)$,
where $\psi(r) = r^{N-\alpha d} (\log\log(1/r))^{\alpha d/N}$.
We conjecture that Theorem \ref{thm2} can be strengthened to
$\mathscr{H}_\varphi(v^{-1}(0) \cap J) = C L(z, J)$ for some constant $C$.

Another result of this paper is a Chung-type law of the iterated logarithm.

\begin{theorem}\label{thm:LIL}
Let $T$ be a compact interval in $\R^N$.
Under Assumptions \ref{a1} and \ref{a2}, 
for any fixed $s \in T$, there exists a constant $K$ that may depend on $s$
such that
\begin{equation}\label{LIL}
\liminf_{r \to 0} \sup_{x \in I_r(s)} \frac{|v(x) - v(s)|}{r^\alpha (\log\log(1/r))^{-\alpha}} 
= K \quad \text{a.s.}
\end{equation}
and $K_0 \le K \le K_1$, where $K_0$ and $K_1$ are positive finite constants 
depending on $T$.
\end{theorem}

The paper is organized as follows.
In Section 2, we review the notions of Hausdorff measure and local time, and some 
lemmas regarding Gaussian probability estimates.
In Section 3, we derive Proposition \ref{prop:chung}, which is a precise probability 
estimate for the upper bound of the Chung-type LIL and 
is a key estimate for the proof of Theorems \ref{thm1} and \ref{thm2}, and 
we also provide the proof of Theorem \ref{thm:LIL} at the end of that section. 
In Section 4, we prove some sojourn time estimates, which are needed in the proof of
Theorem \ref{thm1}.
Then we prove Theorems \ref{thm1} and \ref{thm2} in Sections 5 and 6 respectively.
Finally, in Section 7, we apply the results to the Brownian sheet, fractional Brownian 
sheets and systems of linear stochastic wave equations.

Throughout the paper, unspecific constants are denoted by $c, C, K$, etc., and 
their values can be different in each appearance.

\medskip

\section{Preliminaries}

Consider the class $\Phi$ of nondecreasing, right-continuous functions $\phi : (0, \delta_0) \to (0, \infty)$ for some $\delta_0 > 0$ satisfying $\phi(0+) = 0$
and the condition: there exists $c > 0$ such that $\phi(2r) \le c\phi(r)$ 
for all $0 < r \le \delta_0/2$.
For $\phi \in \Phi$ and $F \subset \R^n$, 
the \emph{$\phi$-Hausdorff measure} of $F$ is defined by
\[ \mathscr{H}_\phi(F) = \lim_{\delta \to 0} \inf\Bigg\{ \sum_{j=1}^\infty \phi(2r_j) : 
F \subset \bigcup_{j=1}^\infty B(x_j, r_j), 0 \le r_j \le \delta \Bigg\}. \]
It is known that $\mathscr{H}_\phi$ is an outer measure and 
every Borel set is $\mathscr{H}_\phi$-measurable.
We say that $\phi$ is the \emph{exact Hausdorff measure function} for $F$ if
$0 < \mathscr{H}_\phi(F) < \infty$.

When $\phi(r) = r^\beta$, where $\beta > 0$, 
the $\phi$-Hausdorff measure is written as 
$\mathscr{H}_\beta$ and is called the $\beta$-dimensional Hausdorff measure.
The Hausdorff dimension of $F$ is defined as
\begin{align*}
\dim F &= \inf\{ \beta > 0 : \mathscr{H}_\beta(F) = 0 \}\\
& = \sup\{ \beta > 0 : \mathscr{H}_\beta(F) = \infty\}.
\end{align*}
The equality between the $\inf$ and the $\sup$ is well known.
We refer to \cite{F, R} for basic properties of Hausdorff measure and Hausdorff 
dimension.

For any finite Borel measure $\mu$ on $\R^n$ and any function $\phi \in \Phi$, 
the \emph{upper $\phi$-density} of $\mu$ at a point $x \in \R^n$ is defined as
\[ \overline{D}^\phi_\mu(x) = \limsup_{r \to 0} \frac{\mu(B(x, r))}{\phi(r)}. \]

The following lemma can be derived from the results in \cite{RT61} (see also \cite{TT85})
and it will be used in the proof of Theorem \ref{thm1}.
For completeness, we give a proof for this lemma.

\begin{lemma}\label{UDT}
For any $\phi \in \Phi$, any finite Borel measure $\mu$ on $\R^n$ 
and any Borel set $F\subset\R^n$,
\begin{equation*}
\mu(F) \le \mathscr{H}_\phi(F) \sup_{x \in F} \overline{D}^\phi_\mu(x).
\end{equation*}
\end{lemma}

\begin{proof}
Let $M = \sup_{x \in F} \overline{D}^\phi_\mu(x)$ and $\eps > 0$.
For each $\delta > 0$, let
\[ F_\delta = \big\{ x \in F : \mu(B(x, r)) \le (M+\eps) \phi(r) \text{ for all } 0 < r \le 2\delta 
\big\}. \]
Suppose the collection $\{B(x_i, r_i) \}_{i=1}^\infty$ of balls with $0 \le r_i \le \delta$ 
is a cover for $F$. In particular it is a cover for $F_\delta$. 
For each ball $B(x_i, r_i)$ containing a point $y_i$ in $F_\delta$, we have
$B(y_i, 2r_i) \supset B(x_i, r_i)$. By the definition of $F_\delta$,
we have $\mu(B(x_i, r_i)) \le \mu(B(y_i, 2r_i)) \le (M+\eps) \phi(2r_i)$.
Then
\[ \mu(F_\delta) \le \sum_{i : B(x_i, r_i) \cap F_\delta \ne \varnothing} \mu(B(x_i, r_i))
\le (M+\eps)\sum_{i=1}^\infty \phi(2r_i). \]
The cover $\{B(x_i, r_i) \}_{i=1}^\infty$ is arbitrary, so we have 
$\mu(F_\delta) \le (M+\eps) \mathscr{H}_\phi(F)$. 
Since $F_\delta \uparrow F$ as $\delta \to 0$ and $\eps > 0$ is arbitrary, 
the result follows.
\end{proof}

Let us recall the definition and basic properties of local times. 
We refer to \cite{GH80} for more details.
Let $T$ be a Borel set in $\R^N$.
The occupation measure of $v$ on $T$ is defined as
\[ \mu_T(A) = \lambda_N\{ x \in T : v(x) \in A \}, \quad A \in \mathscr{B}(\R^d), \]
where $\lambda_N$ denotes the Lebesgue measure on $\R^N$.
If $\mu_T$ is absolutely continuous with respect to the Lebesgue measure $\lambda_d$ 
on $\R^d$, 
we say that the local time of $v$ exists on $T$, and the local time is defined as 
the Radon--Nikodym derivative
\[ L(z, T) = \frac{d\mu_T}{d\lambda_d}(z). \]
It is clear that if the local time exists on $T$, then 
it also exists on any Borel subset of $T$.

Theorem 8.1 of \cite{X} shows that under Assumption \ref{a1}(a),
$v$ has a local time $L(\cdot, T) \in L^2(\lambda_d \times \P)$
if and only if $N > \alpha d$.

Furthermore, by Theorem 8.2 of \cite{X}, if Assumption \ref{a1} is satisfied on 
a compact interval $T = \prod_{j=1}^N[t_j, t_j + h_j]$ and $N > \alpha d$, then 
$v$ has a jointly continuous local time in the sense that
there exists a version of the local time, still denoted by $L(z, \cdot)$, such that a.s.\
$L(z, \prod_{j=1}^N[t_j, t_j+s_j])$ is jointly continuous in all variables 
$(z, s) \in \R^d \times \prod_{j=1}^N[0, h_j]$.

We will use a jointly continuous version of the local time whenever it exists.
If $v$ is continuous and has a jointly continuous local time on $T$, then
$L(z, \cdot)$ defines a Borel measure supported on the level set $v^{-1}(z) \cap T$.
See \cite[Theorem 8.6.1]{A} or \cite[p.12, Remark(c)]{GH80}.

Next, let us recall a small ball probability estimate for Gaussian processes 
that is due to Talagrand \cite{T95}.
The following is a reformulation of this result and its proof can be found in 
Ledoux \cite[p.257]{L}

\begin{lemma}\label{lem:T95}
Let $\{ Z(t): t \in S \}$ be a separable, real-valued, centered Gaussian process
indexed by a bounded set $S$ with the canonical metric 
$d_Z(s, t) = (\E|Z(s) - Z(t)|^2)^{1/2}$.
Let $N_\eps(S)$ denote the  smallest number of $d_Z$-balls of radius $\eps$
that are needed to cover $S$. 
If there is a decreasing function $\psi : (0, \eps_1] \to (0, \infty)$ such that 
$N_\eps(S) \le \psi(\eps)$ for all $\eps \in (0, \eps_1]$ and there are constants 
$a_2 \ge a_1 > 1$ such that
\begin{equation}\label{psi}
a_1 \psi (\eps) \le \psi (\eps/2) \le a_2 \psi (\eps) \quad \text{for all }
\eps \in (0, \eps_1],
 \end{equation}
then for all $u \in (0, \eps_1)$,
\begin{equation*}
\P\left\{ \sup_{s, t \in S} |Z(s) - Z(t)| \le u \right\} \ge \exp \big(-K \psi(u) \big),
\end{equation*}
where $K$ is a constant depending on $a_1$ and $a_2$ only.
\end{lemma}

The next lemma is an isoperimetric inequality for Gaussian processes; 
see \cite[p.302]{LT}.

\begin{lemma}\label{lem:isop}
There is a universal constant $K$ such that the following statement holds.
Let $S$ be a bounded set and $\{ Z(s) : s \in S \}$ be a real-valued Gaussian process.
Let $D = \sup\{ d(s, t) : s, t \in S \}$ be the diameter of $S$ under the 
canonical metric $d_Z$. Then for all $h > 0$,
\begin{equation}
\P\left\{ \sup_{s, t \in S} |Z(s) - Z(t)| \ge K \Big(h + 
\int_0^D \sqrt{\log N_\eps(S)}\, d\eps
\Big) \right\} \le \exp\left( -\frac{h^2}{D^2}\right).
\end{equation}
\end{lemma}

The last lemma of this section provides a probability estimate for the 
uniform modulus of continuity for $v$. It will be needed later in the proofs of Theorems 
\ref{thm1} and \ref{thm2}.

\begin{lemma}\label{lem:mc}
Under Assumption \ref{a1}(a), there exist constants $c$ and $C$ 
such that for all $s \in T$, for all $r > 0$ small, for all $L \ge C$, 
\[ \P\left\{ \sup_{x, y\in I_r(s)}|v(x) - v(y)| \ge L r^\alpha \sqrt{\log(1/r)} \right\} 
\le r^{L^2/c^2}. \]
\end{lemma}

\begin{proof}
With the notations in Lemma \ref{lem:isop}, let $S = I_r(s)$ and $Z(x) = v(x)$.
By Assumption \ref{a1}(a), $d_Z(x, y) \le c_0 |x-y|^\alpha$.
It follows that the diameter of $S$ under $d_Z$ is $D \le C r^\alpha$ and 
$N_\eps(S) \le C\eps^{-N/\alpha}$.
Note that for $0 < x < x_0$ with $x_0 > 0$ small, the inequality 
$\int_0^x \sqrt{\log(1/\eps)}\,d\eps \le Cx \sqrt{\log(1/x)}$ holds
and the function $x \sqrt{\log(1/x)}$ is increasing.
Then for some constant $C$, for $r > 0$ small, we have
\[ \int_0^D \sqrt{\log N_\eps(S)} \, d\eps \le C r^\alpha \sqrt{\log(1/r)}. \]
Therefore, the result now follows from Lemma \ref{lem:isop}.
\end{proof}

\section{A Chung-type law of the iterated logarithm}

The goal of this section is to derive the probability estimate in 
Proposition \ref{prop:chung},
which is a key ingredient of the proofs of Theorems \ref{thm1} and \ref{thm2},
and to provide a proof for the Chung-type LIL of Theorem \ref{thm:LIL}.
First, we need to establish two lemmas.
Recall that $I_r(s) = T \cap \prod_{j=1}^N [s_j - r, s_j + r]$.

\begin{lemma}\label{lem:SB}
Under Assumption \ref{a2}, 
there exist constants $0 < K_0 < \infty$ and $0 < \rho_0 < 1$ such that 
for all compact intervals $I_r(s)$ in $T$ with $r \in (0, \rho_0)$, 
for all $0 < \eps < r^\alpha$ and all $0 < a < b < \infty$,
\begin{equation}\label{SB:u}
\P\left\{ \sup_{x \in I_r(s)} |\tilde v([a, b), x) - \tilde v([a, b), s)| \le \eps \right\} 
\ge \exp\left( - \frac{K_0 r}{\eps^{1/\alpha}}\right).
\end{equation}
\end{lemma}

\begin{proof}
Since $\tilde v$ has i.i.d.\ components, we only need to prove the lemma for $d = 1$.
Recall that $\tilde v([a, b), x) = \sum_{j=1}^N \tilde v^j([a, b), x_j)$.
Fix $s \in T$.
Take $\rho_0 = r_0$ from Assumption \ref{a2}.
We first prove that for each $j$,
\begin{equation}\label{SB:uj}
\P\left\{ \sup_{x_j \in [s_j - r, s_j +r]} |\tilde v^j([a, b), x_j) - \tilde v^j([a, b), s_j)| \le \eps/N\right\} \ge \exp\left(-\frac{K r}{\eps^{1/\alpha}}\right).
\end{equation}
We are going to use Lemma \ref{lem:T95} to prove this inequality 
for fixed $j$, $0 < r < \rho_0$ and $0 < a < b$.
Take $S = [s_j - r, s_j + r]$ and $Z(x_j) = \tilde v^j([a, b), x_j)$. 
By Assumption \ref{a2}(a),
$\tilde v^j([a, b), \cdot)$ and $\tilde v^j(\R_+ \setminus [a, b), \cdot)$ are independent, 
thus
\[
\begin{split}
&{\|\tilde v^j([a, b), x_j) - \tilde v^j([a, b), y_j)\|}_{L^2}^2 
+ {\|\tilde v^j(\R_+ \setminus[a, b), x_j) - \tilde v^j (\R_+ \setminus [a, b), y_j)\|}_{L^2}^2\\
&={\|\tilde v^j(\R_+, x_j) - \tilde v^j(\R_+, y_j)\|}_{L^2}^2.
\end{split}
\]
Then Assumption \ref{a2}(b) implies that for all $x_j, y_j \in S$, 
\begin{equation*}
d_Z(x_j, y_j) \le {\|\tilde v^j(\R_+, x_j) - \tilde v^j(\R_+, y_j)\|}_{L^2}
\le c_2 |x_j - y_j|^{\alpha}.
\end{equation*}
It follows that for all $\eps > 0$ small,
\[N_\eps(S) \le \frac{r}{(c_2^{-1} \eps)^{1/\alpha}} 
= C_{c_2, \alpha}\bigg(\frac{r}{\eps^{1/\alpha}}\bigg). \]
Then we can take $\psi(\eps) = C_{c_2, \alpha}({r}/{\eps^{1/\alpha}})$.
This function satisfies \eqref{psi} with constants $a_1 = a_2 = 2^{1/\alpha}$ 
which are greater than 1. By Lemma \ref{lem:T95}, we can find a constant $K$ 
depending on $a_1$, $a_2$, $c_2$, $\alpha$ and $N$ such that \eqref{SB:uj} is 
satisfied for all $j$, for all $0 < r < \rho_0$ and $0 < \eps < r^{\alpha}$.

To prove \eqref{SB:u}, we use an approximation argument and the 
Gaussian correlation inequality. In fact, for each $j$, we can take 
a sequence $( t_{ij} )_{i=1}^\infty$ that is dense in $[s_j - r, s_j + r]$.
Since $\tilde v^j([a, b), x_j)$ is a.s.\ continuous in $x_j$, 
the probability on the left-hand side of \eqref{SB:u} is
\[ \ge \P\Bigg( \bigcap_{j=1}^N 
\bigg\{ \sup_{x_j \in [s_j - r, s_j + r]}|\tilde v^j([a, b), x_j) - \tilde v^j([a, b), s_j)| \le \eps/N\bigg\}\Bigg) = \lim_{m \to \infty} \P\Bigg( \bigcap_{j=1}^N \{ X^m_j \in B \} \Bigg), \]
where $X^m_j$ is the $m$-dimensional Gaussian vector defined by
\[ X^m_j = \Big(\tilde v^j([a, b), t_{1j}) - \tilde v^j([a, b), s_j), \dots, 
\tilde v^j([a, b), t_{mj}) - \tilde v^j([a, b), s_j)\Big), \quad j = 1, \dots, N, \]
and $B = [-\eps/N, \eps/N]^m$.
Now, for fixed $m$, we apply the Gaussian correlation inequality \cite{R14, LM17} to get that
\[ \P\Bigg( \bigcap_{j=1}^N \{ X^m_j \in B \} \Bigg) \ge \prod_{j=1}^N \P(X^m_j \in B). \]
For each $j$, as $m \to \infty$, $\P(X^m_j \in B)$ converges to the probability on 
the left-hand side of \eqref{SB:uj}. Therefore, we deduce \eqref{SB:u} with $K_0 = NK$.
\end{proof}

\begin{lemma}\label{lem:AE}
Under Assumptions \ref{a1} and \ref{a2},
there exist positive finite constants $K$, $K'$ and $A_0$ such that for any $r > 0$ small,
for any compact interval $I_r(s)$ in $T$, for any $0 < a < b < \infty$, we have
\[ \P\left\{ \sup_{x \in I_r(s)} |v(x) - v(s) - \tilde v([a, b), x) + \tilde v([a, b), s)| \ge h \right\} 
\le \exp\left(-\frac{h^2}{K' A^2}\right) \]
provided that $h \ge K A \sqrt{\log(K r^\alpha/A)}$ and $A \le A_0 r^\alpha$, where 
$A = a^{1-\alpha}r + r^{\gamma_1}b^{\gamma_1-\alpha} + r^{\gamma_2+\alpha}$,
and $0 \le \gamma_1 < \alpha$ and $\gamma_2 > 0$ are the constant in 
Assumption \ref{a2}.
\end{lemma}

\begin{proof}
Since $v(\cdot) - \tilde v([a, b), \cdot)$ has i.i.d.\ components, 
we may assume that $d = 1$.
With the notations in Lemma \ref{lem:isop}, 
set $S = I_r(s)$ and $Z(x) = v(x) - \tilde v([a, b), x)$. 
By Assumption \ref{a2}(c), the canonical metric of $Z$ satisfies
\[ d_Z(x, y) \le C\left(a^{1-\alpha}|x-y| + r^{\gamma_1}b^{\gamma_1-\alpha} 
+ r^{\gamma_2}|x-y|^{\alpha}\right) \]
for all $x, y \in S$. It follows that the diameter of $S$ in $d_Z$ satisfies
\[ D \le C\left(a^{1-\alpha} r + r^{\gamma_1}b^{\gamma_1-\alpha} 
+ r^{\gamma_2+\alpha}\right) = CA. \]
Also, by Assumptions \ref{a1} and \ref{a2}, for all $x, y \in S$,
\begin{align*}
d_Z(x, y) & \le \|v(x) - v(y)\|_{L^2} + 
\sum_{j=1}^N \|\tilde v^j(\R_+, x_j) - \tilde v^j(\R_+, y_j)\|_{L^2} \le C |x - y|^\alpha.
\end{align*}
This implies that 
\[ N_\eps(S) \le C\left(\frac{r^N}{\eps^{N/\alpha}}\right) \]
and hence if $D \le CA_0 r^\alpha$ (which is the case if $A \le A_0 r^\alpha$), there is some large constant $K$ such that
\[ \int_0^D \sqrt{\log N_\eps(S)} \, d\eps \le K \int_0^D \sqrt{\log(K r^\alpha/\eps)} \,d\eps. \]
By the change of variable $\eps = Kr^\alpha \exp(-u^2)$ and the 
elementary inequality $\int_x^\infty u^2 \exp(-u^2) \, du \le C x \exp(-x^2)$
for $x$ large, we deduce that
\[ \int_0^D \sqrt{\log N_\eps(S)} \, d\eps \le K D \sqrt{\log (Kr^\alpha/D)}. \] 
Since $D \le CA$ and the function $f(x) = x\sqrt{\log(r^\alpha/x)}$ 
is increasing for $0 < x\ll r^\alpha$, the right-hand side is
\[ \le KA\sqrt{\log(Kr^\alpha/A)} \]
provided $A \le A_0r^\alpha$, where $A_0$ is a sufficiently small constant.
Therefore, the result now follows from Lemma \ref{lem:isop}.
\end{proof}

\begin{proposition}\label{prop:chung}
Under Assumptions \ref{a1} and \ref{a2},
there exist constants $0 < K_1 < \infty$ and $0 < \rho_0 \le 1$ such that
for any compact interval $I_r(s)$ in $T$, for any $0 < r_0 \le \rho_0$,
\begin{align*}
\P\left\{ \exists\, r \in [r_0^2, r_0], \sup_{x \in I_r(s)}|v(x) - v(s)| \le 
K_1 r^\alpha \left( \log\log\frac1 r\right)^{-\alpha} \right\} 
\ge 1 - \exp\left(-\Big(\log\frac{1}{r_0}\,\Big)^{1/2}\right).
\end{align*}
\end{proposition}

\begin{remark}
This is a key estimate in the proofs of Theorems \ref{thm1} and \ref{thm2}.
Note that the exponent of the $\log\log\frac 1 r$ factor above is different from
the one in Proposition 4.1 of \cite{T95} and the one in Proposition 2.3 of \cite{DMX17}.
\end{remark}

\begin{proof}
The proof follows the idea of Talagrand \cite{T95}.
Let $r_0 > 0$. Fix $U > 1$. The value of $U$ will depend on $r_0$ and 
will be determined later.
For $\ell \ge 1$, define $r_\ell = r_0 U^{-2\ell}$ and $a_\ell = r_0^{-1}U^{2\ell-1}$.
Let $\ell_0$ be the largest integer such that
\begin{equation}\label{l0}
\ell_0 \le \frac{\log(1/r_0)}{2\log U}.
\end{equation}
Then we have $r_0^2 \le r_\ell \le r_0$ for all $1 \le \ell \le \ell_0$.
It suffices to prove that for some constant $K_1$,
\begin{align}\label{chung}
\P\left\{ \exists\, 1 \le \ell \le \ell_0, \sup_{x \in I_{r_\ell}(s)} |v(x) - v(s)|
\le K_1 r_\ell^\alpha \left(\log\log\frac{1}{r_\ell}\right)^{-\alpha}\right\} 
\ge 1 - \exp\left(-\Big(\log \frac{1}{r_0}\,\Big)^{1/2}\right).
\end{align}
By Lemma \ref{lem:SB}, if we take $K_1 = 2(4K_0)^\alpha$, 
then for all $1 \le \ell \le \ell_0$,
\begin{align*}
&\P\left\{ \sup_{x \in I_{r_\ell}(s)}|\tilde v([a_\ell, a_{\ell+1}), x)-\tilde v([a_\ell, a_{\ell+1}), s)| 
\le \frac 1 2 K_1 r_\ell^\alpha\left(\log\log\frac{1}{r_\ell}\right)^{-\alpha} \right\}\\
& \ge \exp\left(-\frac{1}{4}\log\log\frac{1}{r_\ell}\right) 
= \left(\log \frac{1}{r_\ell}\right)^{-1/4}.
\end{align*}
By Assumption \ref{a2}(a), the processes $\tilde v([a_\ell, a_{\ell+1}), \cdot)$, 
$\ell = 1, \dots, \ell_0$, are independent, so
\begin{align}
\begin{aligned}\label{PF}
&\P\left\{ \exists\, 1 \le \ell \le \ell_0, \sup_{x \in I_{r_\ell}(s)} |\tilde v([a_\ell, a_{\ell+1}), x)
- \tilde v([a_\ell, a_{\ell+1}), s)| \le 
\frac 1 2 K_1 r_\ell^\alpha\left(\log\log\frac{1}{r_\ell}\right)^{-\alpha}\right\}\\
& = 1 - \prod_{\ell = 1}^{\ell_0} \left(1 - \P\left\{\sup_{x \in I_{r_\ell}(s)} 
|\tilde v([a_\ell, a_{\ell+1}), x) - \tilde v([a_\ell, a_{\ell+1}), s)| 
\le \frac 1 2 K_1 r_\ell^\alpha\left(\log\log\frac{1}{r_\ell}\right)^{-\alpha}\right\}\right)\\
& \ge 1 - \bigg( 1 - \Big( \log \frac{1}{r_0^2}\,\Big)^{-1/4} \bigg)^{\ell_0}\\
& \ge 1 - \exp\left( - \ell_0 \Big( \log \frac{1}{r_0^2}\,\Big)^{-1/4} \right).
\end{aligned}
\end{align}
Let $A_\ell = a_\ell^{1-\alpha} r_\ell + a_{\ell+1}^{\gamma_1-\alpha} r_\ell^{\gamma_1} 
+ r_\ell^{\gamma_2+\alpha}$,
where $0 \le \gamma_1 < \alpha$ and $\gamma_2 > 0$ are the constants 
given by Assumption \ref{a2}(c).
Then
\[ A_\ell r_\ell^{-\alpha} 
= (a_\ell r_\ell)^{1-\alpha} + (a_{\ell+1}r_\ell)^{-(\alpha-\gamma_1)} 
+ r_\ell^{\gamma_2}. \]
Note that $a_\ell r_\ell = U^{-1}$ and $a_{\ell+1} r_\ell = U$.
Set $\beta = \min\{1-\alpha, \alpha - \gamma_1\} > 0$. Then
\begin{equation}\label{eq:A}
A_\ell r_\ell^{-\alpha} \le 3U^{-\beta}
\end{equation}
provided
\begin{equation}\label{U1}
r_0^{\gamma_2} \le U^{-\beta}.
\end{equation}
In particular, $A_\ell r_\ell^{-\alpha} \le A_0$ if $U$ is large enough.
Set $h_\ell = \frac1 2 K_1 r_\ell^\alpha (\log\log(1/r_\ell))^{-\alpha}$.
By Lemma \ref{lem:AE}, for some constants $K$ and $K'$, we have
\begin{align*}
\P\left\{ \sup_{x \in I_{r_\ell}(s)} |v(x) - v(s) - \tilde v([a_\ell, a_{\ell+1}), x) +
\tilde v([a_\ell, a_{\ell+1}), s)| \ge h_\ell \right\}
\le \exp\left( - \frac{h^2_\ell}{K' A_\ell^2}\right)
\end{align*}
provided $h_\ell \ge K r_\ell^\alpha U^{-\beta}\sqrt{\log(K U)}$, that is, provided
\begin{equation}\label{U2}
U^\beta (\log(K U))^{-1/2} \ge \frac{2K}{K_1} \left(\log\log\frac{1}{r_0}\right)^\alpha,
\end{equation}
which is possible if $U$ is large enough (depending on $r_0$). 
Then by \eqref{eq:A}, we get that
\begin{align}
\begin{aligned}\label{PG}
&\P\left\{ \sup_{x \in I_{r_\ell}(s)} |v(x) - v(s) - \tilde v([a_\ell, a_{\ell+1}), x) +
\tilde v([a_\ell, a_{\ell+1}), s)| \ge 
\frac1 2 K_1 r_\ell^\alpha\left( \log\log\frac{1}{r_\ell}\right)^{-\alpha} \right\}\\
&\le \exp\left( - \frac{U^{2\beta}}{C(\log\log(1/r_0))^{2\alpha}}\right).
\end{aligned}
\end{align}
Let
\begin{align*}
F_\ell &= \left\{ \sup_{x \in I_{r_\ell}(s)} |\tilde v([a_\ell, a_{\ell+1}), x) 
- \tilde v([a_\ell, a_{\ell+1}), s)| \le 
\frac1 2 K_1 r_\ell^\alpha\left( \log\log\frac{1}{r_\ell}\right)^{-\alpha} \right\},\\
G_\ell &= \left\{ \sup_{x \in I_{r_\ell}(s)} |v(x) - v(s) - \tilde v([a_\ell, a_{\ell+1}), x) 
+ \tilde v([a_\ell, a_{\ell+1}), s)| \ge 
\frac1 2 K_1 r_\ell^\alpha\left( \log\log\frac{1}{r_\ell}\right)^{-\alpha} \right\}.
\end{align*}
Then
\begin{align*}
&\P\left\{ \exists\, 1 \le \ell \le \ell_0, \sup_{x \in I_{r_\ell}(s)} |v(x) - v(s)| 
\le K_1 r_\ell^\alpha \left( \log\log\frac{1}{r_\ell}\right)^{-\alpha} \right\}\\
& \ge \P\left( \bigcup_{\ell=1}^{\ell_0} (F_\ell \cap G_\ell^c) \right)\\
& \ge \P\left( \left(\bigcup_{\ell=1}^{\ell_0} F_\ell \right) \cap \left( \bigcap_{\ell=1}^{\ell_0} G_\ell^c \right)\right)\\
& \ge \P\left(\bigcup_{\ell=1}^{\ell_0} F_\ell \right) - \P\left( \bigcup_{\ell=1}^{\ell_0} G_\ell \right).
\end{align*}
By \eqref{PF}, 
\[ \P\left(\bigcup_{\ell=1}^{\ell_0} F_\ell \right) 
\ge 1 - \exp\left( - \ell_0 \Big( \log \frac{1}{r_0^2}\,\Big)^{-1/4} \right), \]
and by \eqref{PG},
\[ \P\left( \bigcup_{\ell=1}^{\ell_0} G_\ell \right) 
\le \ell_0 \exp\left( - \frac{U^{2\beta}}{C(\log\log(1/r_0))^{2\alpha}}\right).\]
It follows that
\begin{align*}
&\P\left\{ \exists\, 1 \le \ell \le \ell_0, \sup_{x \in I_{r_\ell}(s)} |v(x) - v(s)| 
\le K_1 r_\ell^\alpha \left( \log\log\frac{1}{r_\ell}\right)^{-\alpha} \right\}\\
& \ge 1 - \exp\left( - \ell_0 \Big( \log \frac{1}{r_0^2}\,\Big)^{-1/4} \right) 
- \ell_0 \exp\left( - \frac{U^{2\beta}}{C(\log\log(1/r_0))^{2\alpha}}\right).
\end{align*}
Therefore, the proof of \eqref{chung} will be complete provided
\begin{equation}\label{BD:exp}
\exp\left( - \ell_0 \Big( \log \frac{1}{r_0^2}\,\Big)^{-1/4} \right) 
+ \ell_0 \exp\left( - \frac{ U^{2\beta}}{C(\log\log(1/r_0))^{2\alpha}}\right)
\le \exp\left( - \Big(\log\frac{1}{r_0}\,\Big)^{1/2}\right).
\end{equation}
Recall that conditions \eqref{U1} and \eqref{U2} are required for $U$. 
Hence, we can take
\[ U = \left(\log\frac{1}{r_0}\right)^{\frac{1}{2\beta}}. \]
Then for all $r_0$ small enough, by \eqref{l0},
\[ \ell_0 > \frac{\beta}{2} \left(\log\frac{1}{r_0}\right) \left(\log\log\frac{1}{r_0}\right)^{-1} 
>1. \]
Therefore, the left-hand side of \eqref{BD:exp} is bounded above by
\begin{align*}
&\exp\left( - \frac{(\log \frac{1}{r_0})^{3/4}}{C' \log\log \frac{1}{r_0}} \right)
+ \left( \log\frac{1}{r_0}\right) \exp\left( - \frac{\log\frac{1}{r_0}}{C(\log\log\frac{1}{r_0})^{2\alpha}} \right)\\
& \le \exp\left( - \Big( \log\frac{1}{r_0}\Big)^{1/2} \right)
\end{align*}
provided $r_0$ is small enough. This completes the proof of Proposition \ref{prop:chung}.
\end{proof}

\medskip

\begin{proof}[Proof of Theorem \ref{thm:LIL}]
We first prove that for each $s \in T$,
there exists a constant $0 \le K \le \infty$ such that
\begin{equation}\label{0-1law}
\liminf_{r \to 0} \sup_{x \in I_r(s)} \frac{|v(x) - v(s)|}{r^\alpha (\log\log(1/r))^{-\alpha}} 
= K \quad \text{a.s.}
\end{equation}
We claim that the $\liminf$ in \eqref{0-1law} is equal to
\begin{equation}\label{LILclaim}
\liminf_{r \to 0} \sup_{x \in I_r(s)} 
\frac{|\tilde v([a_0, \infty), x) - \tilde v([a_0, \infty), s)|}{r^\alpha(\log\log(1/r))^{-\alpha}}
\quad \text{a.s.}
\end{equation}
Indeed, by Assumption \ref{a2}(c), there exists a constant $C$ such that for $r$ small,
for all $x, y \in I_r(s)$,
\[{\| v(x) - v(y) - \tilde v([a_0, \infty), x) + \tilde v([a_0, \infty), y) \|}_{L^2} 
\le C r^{\gamma} |x - y|^{\alpha}, \]
where $\gamma = \min\{1-\alpha, \gamma_2\} > 0$.
Let $S = I_r(s)$ and $Z(x) = v(x) - \tilde v([a_0, \infty), x)$. 
Then the $d_Z$-diameter of $S$ is 
$D \le C r^{\gamma + \alpha}$, and
by the calculations in Lemma \ref{lem:AE}, we get that for $r$ small,
\[ \int_0^D \sqrt{\log N_\eps(S)} \, d\eps \le KD \sqrt{\log(Kr^\alpha/D)} 
\le K r^{\gamma +\alpha} \sqrt{\log(K/r)}. \]
By taking $r_n = 2^{-n}$ and using Lemma \ref{lem:isop}, 
we see that provided $K$ is large,
\[ \P\Bigg\{ \sup_{x \in I_{r_n}(s)}
|v(x) - v(s) - \tilde v([a_0, \infty), x) + \tilde v([a_0, \infty), s)| 
\ge K r_n^{\gamma +\alpha}\sqrt{\log(1/r_n)} \Bigg\} \le \exp\left(- \frac{K^2 n}
{C^2}\right) \]
and hence by the Borel--Cantelli lemma, 
\[ \limsup_{n \to \infty} \sup_{x \in I_{r_n}(s)}\frac{|v(x) - v(s) - \tilde v([a_0, \infty), x) +
\tilde v([a_0, \infty), s)|}{r_n^{\gamma+\alpha}\sqrt{\log(1/r_n)}} \le K \quad \text{a.s.} \]
This implies that
\[ \lim_{r \to 0} \sup_{x \in I_r(s)}
\frac{|v(x) - v(s) - \tilde v([a_0, \infty), x) + \tilde v([a_0, \infty), s)|}
{r^\alpha(\log\log(1/r))^{-\alpha}} = 0 \quad \text{a.s.}\]
From this, it follows that the $\liminf$ in \eqref{0-1law} is equal to the one in 
\eqref{LILclaim}.

We continue with the proof of \eqref{0-1law}. For each $n \ge 0$, set $a_n = a_0 + n$.
By Assumption \ref{a2}(a),
\[ \tilde v([a_0, \infty), x) = \sum_{n = 0}^\infty \tilde v([a_n, a_{n+1}), x) \quad \text{a.s.} \]
and the Gaussian fields $\tilde v([a_n, a_{n+1}), \cdot)$, $n \ge 0$, are independent.
Consider the $\sigma$-algebras 
$\mathscr{F}_n = \sigma\{ \tilde v([a_m, a_{m+1}), \cdot) : m \ge n \}$ and the 
$\sigma$-algebra $\mathscr{F}_\infty = \bigcap_{n=1}^\infty \mathscr{F}_n$ of all tail events.
By Kolmogorov's zero--one law, $\P(A) = 0$ or $1$ for all $A \in \mathscr{F}_\infty$.
This will imply \eqref{0-1law} if we can show that
the random variable in \eqref{LILclaim} is $\mathscr{F}_\infty$-measurable.

For each $n \ge 1$, we write $\tilde v([a_0, \infty), x) = Y_n(x) + Z_n(x)$, where 
\[ Y_n(x) = \sum_{i=0}^{n-1} \tilde v([a_i, a_{i+1}), x) = \tilde v([a_0, a_n), x) \quad 
\text{and} \quad Z_n(x) = \sum_{i=n}^\infty \tilde v([a_i, a_{i+1}), x). \]
By Assumption \ref{a2}(c), there is a constant $C$ (depending on $n$) such that
for all $x, y \in I_r(s)$,
\begin{align*}
{\|Y_n(x) - Y_n(y)\|}_{L^2}
& \le {\|v(x) - v(y) - \tilde v([a_0, \infty), x) + \tilde v([a_0, \infty), y)\|}_{L^2} \\
& \quad + {\|v(x) - v(y) - \tilde v([a_n, \infty), x) + \tilde v([a_n, \infty), y)\|}_{L^2}\\
& \le Cr^{\gamma}|x - y|^\alpha,
\end{align*}
where $\gamma = \min\{ 1-\alpha, \gamma_2 \}$.
Then, as in the first part of the proof, we use Lemma \ref{lem:isop} to deduce that
\[ \lim_{r \to 0} \sup_{x \in I_r(s)} \frac{|Y_n(x) - Y_n(s)|}{r^\alpha(\log\log(1/r))^{-\alpha}}
= 0 \quad \text{a.s.} \]
It follows the random variable in \eqref{LILclaim} is equal to
\[ \liminf_{r \to 0} \sup_{x \in I_r(s)} 
\frac{|Z_n(x) - Z_n(s)|}{r^\alpha(\log\log(1/r))^{-\alpha}}
\quad \text{a.s.} \]
Since $n \ge 1$ is arbitrary, this random variable is $\mathscr{F}_\infty$-measurable
and therefore constant a.s.~by Kolmogorov's zero--one law.
Combining with the claim in \eqref{LILclaim}, we obtain \eqref{0-1law}
with $0 \le K \le \infty$.

Finally, we apply Proposition \ref{prop:chung} with $r_0 = 2^{-n}$ and 
the Borel--Cantelli lemma to deduce the upper bound $K \le K_1$, while
the lower bound $K \ge K_0$ follows from Theorem 3.4 of \cite{WX11} 
under Assumption \ref{a1}.
\end{proof}

\smallskip

\section{Sojourn time estimates}

In this section, we derive estimates for the sojourn times of $v$.
For each $s \in T$ and $r > 0$, define
\[ \tau(s, r) = \int_T {\bf 1}_{\{ |v(x) - v(s)| \le r \}} dx
= \lambda\{ x \in T : |v(x) - v(s)| \le r\}, \]
which is the sojourn time of $v$ in the ball centered at $v(s)$ with radius $r$.
The following proposition provides upper bounds for the moments of $\tau(s, r)$.

\begin{proposition}\label{prop:ST}
Let Assumption \ref{a1} hold. 
Then there exists a finite constant $K$ such that 
for all $n \ge 1$, for all $x^* \in T$, for all $r > 0$,
\[ \E[\tau(x^*, r)^n] \le K^n (n!)^N r^{nN/\alpha}. \]
\end{proposition}

\begin{proof}
By Fubini's theorem,
\begin{align*}
\E[\tau(x^*, r)^n]
= \int_{T^n} \P\left\{ \max_{1 \le i \le n} |v(x^i) - v(x^*)| \le r \right\} dx^1 \dots dx^n.
\end{align*}
For $m = 1, \dots, n$, define the event $A_m$ by
\[ A_m = \left\{ \max_{1 \le i \le m} |v(x^i) - v(x^*)| \le r\right\}. \]
Then 
\[  \P(A_n) = \E\left[{\bf 1}_{A_{n-1}}
\P\Big\{ |v(x^n) - v(x^*)| \le r\, \Big| \,v(x^*), v(x^1), \dots, v(x^{n-1}) \Big\}\right]. \]
Since $v$ is Gaussian, so is the conditional distribution of $v(x^n)$ given 
$v(x^*), v(x^1), \dots, v(x^{n-1})$ which, by 
sectorial LND [Assumption \ref{a1}(b)],
has conditional variance
$\mathrm{Var}(v(x^n)|v(x^*), v(x^1), \dots, v(x^{n-1}))$ 
$\ge c_1 \sum_{j=1}^N \min_i |x^n_j - x^i_j|^{2\alpha}$,
where the $\min_i$ is taken over $i \in \{ \ast, 0, \dots, n-1 \}$ and $x^0 = 0$.
Then by Anderson's inequality \cite{A55},
\[  \P\Big\{ |v(x^n) - v(x^*)| \le r\, \Big| \,v(x^*), v(x^1), \dots, v(x^{n-1}) \Big\} \le K 
\min\Bigg\{ 1, \frac{r^d}{\big(\sum_{j=1}^N \min_i {|x^n_j - x^i_j|}^{2\alpha}\big)^{d/2}}
\Bigg\}. \]
Note the elementary inequality 
$(\sum_{j=1}^N z_j)^\alpha \le \sum_{j=1}^N z_j^\alpha$ 
for $0 < \alpha < 1$ and $z_j \ge 0$.
Now, let us fix $x^*, x^1, \dots, x^{n-1}$ and estimate the integral
\begin{equation}\label{int}
\int_T \min\Bigg\{ 1, \frac{r^d}{\big(\sum_{j=1}^N \min_i {|x^n_j - x^i_j|}^{2}
\big)^{\alpha d/2}} \Bigg\}\, dx^n.
\end{equation}
We may assume that $x^*, x^1, \dots, x^n$ are distinct, 
because the set of $(x^1, \dots, x^n)$ for which they are not distinct 
has Lebesgue measure 0.
Consider the following set consisting of $(n+1)^N$ points:
\[ Y = \prod_{j=1}^N \left\{ x^*_j, x^0_j, x^1_j, \dots, x^{n-1}_j \right\}. \]
We use these points to produce a partition of $T$ into subintervals $S$ such that
for each subinterval, there is one and only one $y \in Y$ such that
\begin{equation}\label{y}
\min_i |x^n_j - x^i_j| = |x^n_j - y_j|
\end{equation}
for all $x^n \in S$ and all $j = 1, \dots, N$.
This can be done by first ordering the points with permutations
$\Pi_1, \dots, \Pi_N$ on the set $\{ \ast, 0, \dots, n-1 \}$ such that for every $j$,
\[ x_j^{\Pi_j(\ast)} < x_j^{\Pi_j(0)} < \dots < x_j^{\Pi_j(n-1)}, \]
and then using the perpendicular bisectors of every two adjacent points to construct 
the partition. There are at most $(n+1)^N = O(n^N)$ subintervals $S$.
Then the integral in \eqref{int} is equal to
\begin{align*}
\sum_{S}\int_{S} \min\Bigg\{ 1, \frac{r^d}{\big(\sum_{j=1}^N \min_i {|x^n_j - x^i_j|}^{2}\big)^{\alpha d/2}} \Bigg\}\, dx^n.
\end{align*}
For each fixed subinterval $S$, we use \eqref{y} and polar coordinates 
$x^n - y = \rho \theta$ to get that
\begin{align*}
\int_{S} \min\Bigg\{ 1, \frac{r^d}{\big(\sum_{j=1}^N {|x^n_j - y_j|}^{2}\big)^{\alpha d/2}}
\Bigg\}\, dx^n
&\le K \int_0^\infty \min\left\{ 1, {r^d}{\rho^{-\alpha d}} \right\} \rho^{N-1}\, d\rho\\
&= K \int_0^{r^{1/\alpha}} \rho^{N-1}\, d\rho + K
\int_{r^{1/\alpha}}^\infty r^d \rho^{N-\alpha d - 1} \, d\rho\\
& = K r^{N/\alpha}.
\end{align*}
It follows that
\begin{align*}
\E[\tau(x^*, r)^n] 
&= \int_{T^{n-1}} \P(A_{n-1}) dx^1 \dots x^{n-1}
\int_T \min\Bigg\{ 1, \frac{r^d}{\big(\sum_{j=1}^N {\min_i |x^n_j - x^i_j|}^{2}\big)^{\alpha d/2}} \Bigg\}\, dx^n\\
& \le K n^N r^{N/\alpha} \int_{T^{n-1}} \P(A_{n-1}) dx^1 \dots dx^{n-1}\\
& = K n^N r^{N/\alpha} \E[\tau(x^*, r)^{n-1}].
\end{align*}
Therefore, the result can be deduced by induction.
\end{proof}

From the moment estimates, we can derive the following asymptotic result for the
sojourn times, which will be used in the proof of Theorem \ref{thm1}.

\begin{proposition}\label{prop:UD}
Under Assumption \ref{a1}, for all $s \in T$,
\[ \limsup_{r \to 0} \frac{\tau(s, r)}{\phi(r)} \le K \quad \text{a.s.}, \]
where $\phi(r) = r^{N/\alpha}(\log\log(1/r))^N$ and $K$ is the constant in Proposition \ref{prop:ST}.
\end{proposition}

\begin{proof}
Consider a constant $C > K$.
By Fubini's theorem and Jensen's inequality,
\begin{align*}
\E[\exp(C^{-1/N} r^{-1/\alpha}\tau(s, r)^{1/N})]
& = \sum_{n=0}^\infty \frac{1}{n!} C^{-n/N} r^{-n/\alpha} \E[\tau(s, r)^{n/N}]\\
&\le \sum_{n=0}^\infty \frac{1}{n!} C^{-n/N} r^{-n/\alpha} \E[\tau(s, r)^n]^{1/N}.
\end{align*}
Then by Proposition \ref{prop:ST},
\begin{align*}
\E[\exp(C^{-1/N}r^{-1/\alpha}\tau(s, r)^{1/N})]
\le \sum_{n=0}^\infty (K/C)^{n/N} =: A,
\end{align*}
where $A$ is finite since $C > K$.
It follows that for any $\eps > 0$,
\begin{align*}
&\P\left\{ \tau(s, r) > (1+\eps)^N C r^{N/\alpha}(\log\log(1/r))^N\right\}\\
&= \P\left\{ \exp(C^{-1/N} r^{-1/\alpha}\tau(s, r)^{1/N}) > (\log(1/r))^{1+\eps}\right\}\\
& \le \frac{A}{(\log(1/r))^{1+\eps}}.
\end{align*}
Set $r_n = e^{-n/\log n}$. Then by the Borel--Cantelli lemma, we have
\[ \limsup_{n \to \infty} \frac{\tau(s, r_n)}{\phi(r_n)} \le (1+\eps)^N C \quad \text{a.s.} \]
From this, we can use a monotonicity argument to deduce that
\[ \limsup_{r \to 0} \frac{\tau(s, r)}{\phi(r)} \le (1+\eps)^N C \quad \text{a.s.} \]
Finally, we obtain the result by letting
$\eps \to 0$ and $C \downarrow K$ along rational sequences.
\end{proof}

\medskip

\section{Proof of Theorem \ref{thm1}}

It suffices to prove that for each compact interval $J$ in $T$ with rational vertices, 
\begin{equation}\label{thm1eq}
\P\Big\{C_1\lambda(J) \le \mathscr{H}_\phi(v(J)) \le C_2\lambda(J)\Big\} = 1.
\end{equation}
Then \eqref{thm1eq0} follows by taking sequences $J_n$ of compact intervals with
rational vertices such that $J_n \uparrow J$ and by the continuity of $v$.

The proof of \eqref{thm1eq} is similar to the proof in \cite{X96}.
We start with the proof of the lower bound.
Let $J$ be any compact interval in $T$.
Define a random Borel measure $\mu$ on $\R^N$ by
\[\mu(B) = \lambda\{ s \in J : v(s) \in B\}\]
for any Borel set $B$ in $\R^N$.
Note that $\mu(\R^d) = \lambda(J)$ and $\mu$ is a finite measure.
Consider the random set
\[ F = \left\{ v(s) : s \in J \text{ and } 
\limsup_{r \to 0} \frac{\mu(B(v(s), r))}{\phi(r)} \le K \right\}, \]
where $K$ is the constant in Proposition \ref{prop:UD}.
Then by Lemma \ref{UDT}, we have
\[ \mu(F) \le \mathscr{H}_\phi(F) \sup_{s \in F} \overline{D}^\phi_\mu(s) \le K \mathscr{H}_\phi(v(J)). \]
It remains to prove that $\mu(F) = \lambda(J)$ a.s.
To see this, note that $\mu = \lambda \circ v^{-1}$.
Then by the change of variable formula and Fubini's theorem,
\begin{align*}
\E[\mu(F)] & = \E \int_{\R^d} {\bf 1}_F(z) \mu(dz) = \E \int_J {\bf 1}_F(v(s))\, ds\\
& = \int_J\, \P \bigg\{ \limsup_{r \to 0} \frac{\mu(B(v(s), r))}{\phi(r)} \le K \bigg\} \, ds.
\end{align*}
Then by Proposition \ref{prop:UD}, we have $E[\mu(F)] = \lambda(J)$.
This implies that $\mu(F) = \lambda(J)$ a.s.\ and hence
\[ \mathscr{H}_\phi(v(J)) \ge K^{-1} \lambda(J) \quad \text{a.s.} \]

We turn to the proof of the upper bound in \eqref{thm1eq}.
The idea is due to Talagrand \cite{T95}.
For each $p \ge 1$, consider the random set
\begin{align*}
R_p = \left\{ s \in J: \exists\, r \in [2^{-2p}, 2^{-p}] \text{ such that }
\sup_{x \in I_r(s)}|v(x) - v(s)| \le K_1 r^\alpha\left(\log\log\frac1 r\right)^{-\alpha} \right\}
\end{align*}
and the event
\[ \Omega_{p, 1} 
= \Big\{ \omega : \lambda(R_p) \ge \lambda(J)(1 - \exp(-\sqrt{p}/4)) \Big\}. \]
Note that
\[ \Omega_{p, 1}^c = 
\Big\{ \omega : \lambda(J\setminus R_p) > \lambda(J) \exp(-\sqrt{p}/4) \Big\}. \]
Then by Markov's inequality,
\begin{align*}
\P(\Omega_{p, 1}^c) \le \frac{\E[\lambda(J\setminus R_p)]}{\lambda(J) \exp(-\sqrt{p}/4)}.
\end{align*}
By Fubini's theorem and Proposition \ref{prop:chung}, the numerator is equal to
\begin{align*}
\E\int_J {\bf 1}_{J \setminus R_p}(s) \, ds = \int_J \P(s \not\in R_p) \, ds
\le \lambda(J) \exp(-\sqrt{p/2}).
\end{align*}
It follows that 
$\sum_{p=1}^\infty \P(\Omega_{p, 1}^c) \le \sum_{p=1}^\infty \exp(-\sqrt{p}/4) < \infty$.

We will call an interval of the form $\prod_{j=1}^N[m_j 2^{-p}, (m_j+1)2^{-p})$, where
$m_j \in \Z$, a dyadic cube of order $p$.
Let $\mathscr{C}_p$ be the set of all dyadic cubes of order $p$ that intersect $J$,
and for convenience, we replace each $C \in \mathscr{C}_p$ by the intersection
$C \cap J$.
Let $K_2$ be a constant and consider the event
\[ \Omega_{p, 2} = \bigg\{ \omega : \forall\, C \in \mathscr{C}_{2p}, 
\sup_{x, y \in C} |v(x) - v(y)| \le K_2 2^{-2p \alpha}\sqrt{p} \bigg\}. \]
Since the number of cubes in $\mathscr{C}_{2p}$ is $\le K 2^{2pN}$,
it follows from Lemma \ref{lem:mc} that 
$\sum_{p=1}^\infty\P(\Omega_{p, 2}^c) < \infty$ provided $K_2$ is large enough.
Let $\Omega_p = \Omega_{p, 1} \cap \Omega_{p, 2}$.
Then with probability 1, $\Omega_p$ occurs for all $p$ large.
Let $\omega \in \Omega_p$.
We are going to construct a random cover for $v(J)$.

We say that $C$ is a ``good'' dyadic cube of order $q$ if
\begin{equation}\label{good}
\sup_{x, y \in C} |v(x) - v(y)| \le 
4K_1 2^{-q\alpha} (\log\log 2^{q})^{-\alpha}.
\end{equation}
Then for each $s \in R_p$ ($p$ large), there is at least one good dyadic cube 
of order $q$ ($p \le q \le 2p$) that contains $s$, and we can choose such a good
dyadic cube $C$ of the smallest order.
Indeed, if $s \in R_p$, then there is $r \in [2^{-q}, 2^{-q+1}]$ with $p + 1 \le q \le 2p$
such that
\[ \sup_{x \in I_r(s)} |v(x) - v(s)| \le K_1 r^\alpha(\log\log(1/r))^{-\alpha}. \]
Then for the dyadic cube $C$ of order $q$ that contains $s$,
we can verify that \eqref{good} is satisfied:
\begin{align*}
\sup_{x, y \in C} |v(x) - v(y)| &\le \sup_{x \in C}|v(x) - v(s)| + \sup_{y \in C}|v(y) - v(s)|\\
&\le 2K_1 (2^{-q+1})^\alpha (\log\log 2^{q-1})^{-\alpha}\\
& \le 4K_1 2^{-q\alpha} (\log\log 2^q)^{-\alpha}.
\end{align*}
Therefore, we can obtain a family $\mathcal{H}^1_p$ of disjoint 
good dyadic cubes of order between $p$ and $2p$ that cover the set $R_p$. 

Next, we let $\mathcal{H}^2_p$ be the family of dyadic cubes $C\in \mathscr{C}_{2p}$
of order $2p$ that are disjoint from the cubes of $\mathcal{H}^1_p$ and intersect $J$.
We call these ``bad'' dyadic cubes.
Then the cubes of $\mathcal{H}^2_p$ are contained in $J \setminus R_p$ and
the family $\mathcal{H}_p := \mathcal{H}^1_p \cup \mathcal{H}^2_p$ covers $J$.

For each dyadic cube $C$, let $s_C$ be the lower-left vertex of $C$ and
\begin{equation}\label{rC}
r_C = \begin{cases}
4K_1 2^{-q\alpha} (\log\log 2^q)^{-\alpha} & \text{if } C \in \mathcal{H}^1_p \text{ and } C \text{ is of order } q\\
K_2 2^{-2p \alpha}\sqrt{p} & \text{if } C \in \mathcal{H}^2_p.
\end{cases}
\end{equation}
Now, for $\omega \in \Omega_p$, 
define a family $\mathcal{F}_p$ of Euclidean balls in $\R^d$ by 
\[ \mathcal{F}_p = 
\left\{ B(v(s_C), r_C) : C \in \mathcal{H}^1_p \cup \mathcal{H}^2_p \right\}. \]
Note that for $C \in \mathcal{H}^1_p$, the image $v(C)$ is contained in 
the ball $B(v(s_C), r_C)$ because of \eqref{good}, and 
this is also true for $C \in \mathcal{H}^2_p$ since $\omega \in \Omega_{p, 2}$.
Hence, $\mathcal{F}_p$ is a random cover for $v(J)$ 
as long as $\omega \in \Omega_p$.

Recall that $\phi(r) = r^{N/\alpha} (\log\log(1/r))^N$.
If $C$ is a good dyadic cube in $\mathcal{H}^1_p$, then
\[ \phi(2r_C) \le K [2^{-q\alpha} (\log\log 2^q)^{-\alpha}]^{N/\alpha} (\log\log 2^q)^N
\le K 2^{-qN}. \]
Since the cubes of $\mathcal{H}^1_p$ are disjoint and contained in $J$, it follows that
\begin{align*}
\sum_{C \in \mathcal{H}^1_p} \phi(2r_C) 
\le K \sum_{C \in \mathcal{H}^1_p} \lambda(C) \le K \lambda(J).
\end{align*}
If $C$ is a bad dyadic cube in $\mathcal{H}^2_p$, then
\[ \phi(2r_C) \le K 2^{-2pN} p^{N/(2\alpha)}(\log p)^N. \]
Since $\omega \in \Omega_{p, 1}$, the number of cubes in $\mathcal{H}^2_p$ is at most
\[ 2^{2pN} \lambda(J\setminus R_p) \le 2^{2pN} \lambda(J) \exp(-\sqrt{p}/4). \]
It follows that
\[ \sum_{C \in \mathcal{H}^2_p} \phi(2r_C) 
\le K \lambda(J) \exp(-\sqrt{p}/4) \,p^{N/(2\alpha)}(\log p)^N. \]
Therefore, with probability 1,
\begin{align*}
\mathscr{H}_\phi(v(J))
& \le \liminf_{p \to \infty} \left[ {\bf 1}_{\Omega_p} 
\sum_{C \in \mathcal{H}^1_p \cup \mathcal{H}^2_p} \phi(2r_C) \right]\\
& \le \liminf_{p \to \infty} \left[ K \lambda(J) \left(1 + \exp(-\sqrt{p}/4) \,p^{N/(2\alpha)}(\log p)^N\right)\right]\\
& = K \lambda(J).
\end{align*}
The proof of Theorem \ref{thm1} is complete. \qed

\medskip

\section{Proof of Theorem \ref{thm2}}

First, consider the case $N > \alpha d$. 
For the lower bound in \eqref{thm2eq0}, it suffices to prove that there exists 
a positive finite constant $C$ such that for each $z \in \R^d$,
for each compact interval $J$ in $T$ with rational vertices, we have
\begin{equation*}
\P\Big\{ C L(z, J) \le \mathscr{H}_\varphi(v^{-1}(z) \cap J) \Big\} = 1.
\end{equation*}
This is a consequence of Theorem 4.6 of \cite{WX11} under Assumption \ref{a1}.
Since the local time $L(z, J)$ of $v$ is continuous in $J$ \cite[Theorem 8.2]{X}, 
we can then take, for any $J \in \mathscr{I}(T)$, a sequence $J_n$ of 
compact intervals with rational vertices such that $J_n \uparrow J$ to deduce that
\[ \P\Big\{ C L(z, J) \le \mathscr{H}_\varphi(v^{-1}(z) \cap J) \text{ for all } 
J \in \mathscr{I}(T)\Big\} = 1. \]
To show that the Hausdorff measure above is finite simultaneously for all $J$, 
we will prove that
\begin{equation}\label{EH}
\E[\mathscr{H}_\varphi(v^{-1}(z) \cap J)] \le C \lambda(J)
\end{equation}
for all $z \in \R^d$ and all sufficiently small compact intervals $J$ in $T$.
The proof of \eqref{EH} is similar to the proof of Theorem 4.2 in \cite{X97}
and is based on the method of \cite{T98}.

Let $\eps_0$ be the constant given by Assumption \ref{a3}.
Fix a small constant $0 < \rho_0 \le \eps_0$ whose value will be determined. 
Let $0 < \rho \le \rho_0$. 
We take $J$ to be an interval $I_\rho(x) = T \cap \prod_{j=1}^N[x_j - \rho, x_j + \rho]$
and let $x' \in J^{(c\rho)}$ be given by Assumption \ref{a3}.
We claim that \eqref{EH} holds for such an interval $J$.

Let us define
\[ v^2(y) = \E(v(y)| v(x')) \quad \text{and } \quad  v^1(y) = v(y) - v^2(y).\]
Then the Gaussian random fields 
$v^1 = \{v^1(y) : y \in J\}$ and $v^2 = \{v^2(y) : y \in J\}$ are independent.
By Lemma 5.3 of \cite{DMX17}, under Assumption \ref{a3},
$v^2$ has a continuous version and
there exists a finite constant $K_3$ and $\alpha < \delta \le 1$ such that for all $y, \bar y \in J$,
\begin{equation}\label{v2}
|v^2(y) - v^2(\bar y)| \le K_3 |v(x')| |y - \bar y|^\delta.
\end{equation}

We define the random set $R_p$ and the events $\Omega_{p, 1}$ and $\Omega_{p, 2}$
as in the proof of Theorem \ref{thm1}.
In addition, we fix a constant $\beta$ such that $0 < \beta < \delta - \alpha$ 
and consider the event
\[ \Omega_{p, 3} = \big\{ |v(x')| \le 2^{p\beta} \big\}. \]
By \eqref{v2}, the H\"older constant of $v^2$ is not too large on $\Omega_{p, 3}$.
Since $v(x')$ is a Gaussian vector, we have 
$\sum_{p=1}^\infty \P(\Omega_{p, 3}^c) < \infty$.
Let $\Omega_p = \Omega_{p, 1} \cap \Omega_{p, 2} \cap \Omega_{p, 3}$.
Then with probability 1, $\Omega_p$ occurs for all $p$ large.

To prove \eqref{EH}, we let $\omega \in \Omega_p$ and construct an efficient covering 
(depending on $\omega$) for the level set $v^{-1}(z) \cap J$.
We say that $C$ is a ``good'' dyadic cube of order $q$ if 
condition \eqref{good} with $v$ replaced by $v^1$ is satisfied, i.e.,
\begin{equation}\label{good'}
\sup_{y, \bar y \in C} |v^1(y) - v^1(\bar y)| \le 4K_1 2^{-q\alpha} (\log\log 2^q)^{-\alpha}.
\end{equation}
As in the proof of Theorem \ref{thm1}, every point $s \in R_p$ is contained in
some dyadic cube $C$ of order $q$ ($p \le q \le 2p$) such that
\[ \sup_{y, \bar y \in C}|v(y) - v(\bar y)| 
\le 2K_1 (2^{-q+1})^\alpha (\log\log 2^{q-1})^{-\alpha}. \]
By $\omega \in \Omega_{p, 3}$ and \eqref{v2},
\[
\sup_{y, \bar y \in C}|v^2(y) - v^2(\bar y)| 
\le K_3 2^{q\beta}(\sqrt{N}2^{-q})^\delta 
\le K 2^{-q(\delta-\beta)}.
\]
Since $v = v^1 + v^2$ and $\beta$ is chosen such that $\alpha < \delta - \beta$, 
this implies \eqref{good'} for $p$ large, hence there is at least one good 
dyadic cube containing $s$ of order between $p$ and $2p$.
Take such a good dyadic cube of the smallest order.
Then we define $\mathcal{H}_p = \mathcal{H}^1_p \cup \mathcal{H}^2_p$,
where $\mathcal{H}^1_p$ is the family of disjoint good dyadic cubes
of order between $p$ and $2p$ covering $R_p$ obtained in this way, 
and $\mathcal{H}^2_p$ is the family of bad dyadic cubes of order $2p$ 
that are disjoint from the cubes of $\mathcal{H}^1_p$ and intersect $J$.
It follows that the cubes of $\mathcal{H}^2_p$ are contained in $J \setminus R_p$ and
$\mathcal{H}_p$ covers $J$.
Note that the families $\mathcal{H}^1_p$ and $\mathcal{H}^2_p$ depend on
the process $v^1$ only. This property will be used later in the proof.

For each dyadic cube $C$, define $s_C$ as
the lower-left vertex of $C$ and $r_C$ as in \eqref{rC}.
Now, we construct the family $\mathcal{G}_p = \mathcal{G}_p(\omega)$ of cubes 
by setting
\[ \mathcal{G}_p = 
\{ C \in \mathcal{H}_p : \omega \in \Omega_C \}, \]
where
\[ \Omega_C = \{ |v(s_C) - z| \le 2r_C \}. \]

We verify that $\mathcal{G}_p$ covers $v^{-1}(z) \cap J$ for $\omega \in \Omega_p$
and $p$ large.
Let $y \in v^{-1}(z) \cap J$. Then $v(y) = z$ and $y \in C$ 
for some dyadic cube $C \in \mathcal{H}_p$ because $\mathcal{H}_p$ covers $J$.
We need to show that $C$ is a member of $\mathcal{G}_p$.
Consider the following two cases.

Case 1. If $C$ is a good dyadic cube in $\mathcal{H}^1_p$ of order $q$, 
where $p \le q \le 2p$, then condition \eqref{good'} is satisfied and thus
\[ |v^1(s_C) - v^1(y)| \le r_C. \]
Since $\omega \in \Omega_{p, 3}$, by \eqref{v2}, 
\[ |v^2(s_C) - v^2(y)| \le K 2^{-q(\delta - \beta)}, \]
which is $\le 4K_1 2^{-q\alpha}(\log\log 2^q)^{-\alpha} = r_C$ for $p$ large
because $\alpha < \delta - \beta$.
Since $v = v^1 + v^2$, it follows that
\[ |v(s_C) - z| = |v(s_C) - v(y)| \le 2r_C. \]
Hence $\Omega_C$ occurs and $C \in \mathcal{G}_p$.

Case 2. If $C$ is a bad dyadic cube in $\mathcal{H}^2_p$, then $C$ of order $2p$. 
Since $\omega \in \Omega_{p, 2}$, we have
\[ |v(s_C) - z| = |v(s_C) - v(y)| \le r_C \]
and hence $C \in \mathcal{G}_p$. 
Therefore, $\mathcal{G}_p$ covers $v^{-1}(z) \cap J$ for $\omega \in \Omega_p$
and $p$ large.

Consider the $\sigma$-algebra $\Sigma_1 = \sigma\{ v^1(y) : y \in J \}$.
In order to estimate $\E[\mathscr{H}_\varphi(v^{-1}(z) \cap J)]$ 
using the random cover $\mathcal{G}_p$,
we need to first estimate the conditional probability $\P(\Omega_C|\Sigma_1)$.
Note that $\mathrm{Var}(v(y)|v(x')) \le \E[(v(y) - v(x'))^2]$.
Then by the conditional variance formula and
Assumptions \ref{a1}(a) and \ref{a3}(a), we get that for all $y \in J$,
\begin{align*}
\mathrm{Var}(v^2(y)) &= \mathrm{Var}(\E(v(y)|v(x')))\\
&= \mathrm{Var}(v(y)) - \E[\mathrm{Var}(v(y)|v(x'))]\\
& \ge c_3^2 - c_0^2 |y - x'|^{2\alpha}.
\end{align*}
Since $J$ has side length $2\rho$ and $x' \in J^{(c\rho)}$, with $0 < \rho \le \rho_0$,
we can choose $\rho_0 > 0$ to be small enough so that
the variance of $v^2(y)$ is bounded below by a positive constant.
It follows that $\P(|v^2(y) - a| \le r) \le K r^d$ for all $y \in J$, $a \in \R^d$ and $r > 0$,
where $K$ is a constant. 
Recall that $\mathcal{H}^1_p$ and $\mathcal{H}^2_p$ depend on $v^1$ only, so
$r_C$ is $\Sigma_1$-measurable.
Hence, by the independence of $v^1$ and $v^2$,
\begin{equation}\label{condp}
\P(\Omega_C|\Sigma_1) \le K r_C^d.
\end{equation}

Now, we estimate $\E[\mathscr{H}_\varphi(v^{-1}(z) \cap J)]$ using the cover $\mathcal{G}_p$.
Recall that the cubes of $\mathcal{H}^2_p$ are contained in $J \setminus R_p$,
and note that on $\Omega_{p, 1}$, the set $J \setminus R_p$ has Lebesgue measure 
$\le \lambda(J) \exp(-\sqrt{p}/4)$. 
It follows that $\Omega_p$ is contained in the event $\tilde\Omega_p$ 
that the number of cubes in $\mathcal{H}^2_p$ is 
at most $\lambda(J) 2^{2pN} \exp(-\sqrt{p}/4)$. 
Then
\begin{align*}
\E\Bigg[{\bf 1}_{\Omega_p} \sum_{C \in \mathcal{G}_p} \varphi(\mathrm{diam}(C))\Bigg]
& \le \E\Bigg[ {\bf 1}_{\tilde\Omega_p}\sum_{C \in \mathcal{H}_p} 
\varphi(\mathrm{diam}(C)) {\bf 1}_{\Omega_C}\Bigg].
\end{align*}
Since $\tilde\Omega_p$ and $\mathcal{H}_p$ are measurable with respect to 
$\Sigma_1$, 
we can take conditional expectation on $\Sigma_1$ and use \eqref{condp} to 
get
\begin{align*}
\E\Bigg[ {\bf 1}_{\tilde\Omega_p}\sum_{C \in \mathcal{H}_p} 
\varphi(\mathrm{diam}(C)) \P(\Omega_C|\Sigma_1)\Bigg]
\le K\E\Bigg[ {\bf 1}_{\tilde\Omega_p}\sum_{C \in \mathcal{H}_p} 
\varphi(\mathrm{diam}(C)) r_C^d\Bigg].
\end{align*}
Recall that $\varphi(r) = r^{N-\alpha d} (\log\log(1/r))^{\alpha d}$.
If $C \in \mathcal{H}^1_p$ is a good dyadic cube of order $q$, then
\[ \varphi(\mathrm{diam}(C)) r_C^d 
\le K 2^{-q(N-\alpha d)}(\log q)^{\alpha d}(2^{-q\alpha}(\log q)^{-\alpha})^d \le K 2^{-qN}. \]
It follows that
\[ \E\Bigg[ \sum_{C \in \mathcal{H}^1_p} \varphi(\mathrm{diam}(C)) r_C^d\Bigg] 
\le K \sum_{C \in \mathcal{H}^1_p} \lambda(C) \le K \lambda(J). \]
If $C \in \mathcal{H}^2_p$ is a bad dyadic cube, then 
\[ \varphi(\mathrm{diam}(C))r_C^d 
\le K2^{-2pN}(\log p)^{\alpha d} p^{d/2}. \]
Since the number of cubes in $\mathcal{H}^2_p$ is 
$\le \lambda(J) 2^{2pN} \exp(-\sqrt{p}/4)$ on $\tilde \Omega_p$, it follows that
\[ \E\Bigg[ {\bf 1}_{\tilde\Omega_p}
\sum_{C \in \mathcal{H}^2_p} \varphi(\mathrm{diam}(C)) r_C^d\Bigg] 
\le K \lambda(J) \exp(-\sqrt{p}/4)(\log p)^{\alpha d} p^{d/2}. \]
Therefore, by Fatou's lemma,
\begin{align*}
\E[\mathscr{H}_\varphi(v^{-1}(z) \cap J)] 
&\le \liminf_{p \to \infty} \,\E\Bigg[ {\bf 1}_{\Omega_p} 
\sum_{C \in \mathcal{G}_p} \varphi(\mathrm{diam}(C)) \Bigg]\\
& \le  \liminf_{p \to \infty} \left[ K \lambda(J)\left( 1 + \exp(-\sqrt{p}/4)(\log p)^{\alpha d} p^{d/2} \right)\right]\\
& = K \lambda(J).
\end{align*}
This completes the proof of \eqref{EH} and establishes \eqref{thm2eq0}.

Consider the case $N \le \alpha d$.
Suppose, towards a contradiction, the event  
$\Omega' := \{v^{-1}(z) \cap J \ne \varnothing\}$ has positive probability.
Then the event $\Omega' \cap \Omega^* 
= \bigcup_{n \ge 1} \bigcap_{p \ge n} (\Omega' \cap \Omega_p)$ 
also has positive probability.
Consider the random variable
\[ X = \liminf_{p \to \infty} {\bf 1}_{\Omega' \cap \Omega_p} 
\sum_{C \in \mathcal{G}_p}\varphi(\mathrm{diam}(C)). \]
By the calculations above, $\E(X) \le K\lambda(J) < \infty$.
On the other hand, for $p$ large,
since $\mathcal{G}_p$ covers $v^{-1}(z) \cap J$ on $\Omega_p$ 
and $v^{-1}(z) \cap J \ne \varnothing$ on $\Omega'$, 
we see that $\mathcal{G}_p$ is nonempty on $\Omega' \cap \Omega_p$.
Moreover, $N \le \alpha d$ implies that $\varphi(r) \to \infty$ as $r \to 0$. 
It follows that $X = \infty$ on $\Omega' \cap \Omega^*$, which contradicts the fact that
$\E(X) < \infty$. Therefore, we have $v^{-1}(z) \cap J \ne \varnothing$ a.s.\ for all 
sufficiently small compact intervals $J$ in $T$. 
This implies $v^{-1}(z) \cap T \ne \varnothing$ a.s.\ and 
completes the proof of Theorem \ref{thm2}.
\qed

\medskip

\section{Examples}

\subsection{Brownian sheet and fractional Brownian sheets}
Let $v = \{v(x) : x \in \R^N_+\}$ be a fractional Brownian sheet 
from $\R^N_+$ to $\R^d$ with Hurst indices $H_1 = \dots = H_N = \alpha$
where $0 < \alpha < 1$, i.e., a centered, continuous Gaussian random field 
with i.i.d.\ components $v_1, \dots, v_d$ and covariance
\[ \E[v_i(x)v_i(y)] = 
\prod_{j=1}^N \frac 1 2 (|x_j|^{2\alpha} + |y_j|^{2\alpha} - |x_j - y_j|^{2\alpha}).
\]
When $\alpha = 1/2$, $v$ is a Brownian sheet from $\R^N_+$ to $\R^d$.

The uniform dimension of the level sets of the Brownian sheet was studied by Adler 
\cite{A78}, and that of the images was studied by Mountford \cite{M99} and 
Khoshnevisan et al.\ \cite{KWX06}.
The Hausdorff measures of the range and graph of the Brownian sheet were studied 
by Ehm \cite{E81}.
For fractional Brownian sheets, dimension results for the images and level sets 
can be found in \cite{AX05} and \cite{WX07}.

It is known that for any compact interval $T$ in $(0, \infty)^N$, there exists a 
positive finite constant $c_0$ such that for all $x, y \in T$,
\begin{equation}\label{fBs1}
c_0^{-1} \sum_{j=1}^N |x_j - y_j|^{2\alpha} \le \E(|v(x) - v(y)|^2) 
\le c_0 \sum_{j=1}^N|x_j - y_j|^{2\alpha},
\end{equation}
and $v$ satisfies sectorial LND:
there exists a positive constant $c_1$ such that for all integers $n \ge 1$, 
for all $x, y^1, \dots, y^n \in T$,
\begin{equation}\label{fBsLND}
\mathrm{Var}(v_1(x)|v_1(y^1), \dots, v_1(y^n)) \ge 
c_1 \sum_{j=1}^N \min_{0 \le i \le n} |x_j - y_j^i|^{2\alpha},
\end{equation}
where $y^0 = 0$. See \cite{AX05, WX07}.

Let us recall the \emph{harmonizable representation} for the fractional Brownian sheet 
\cite[\S 5.1]{DLMX}.
For any $x, y \in \R$, we have the identity
\[ \frac 1 2 (|x|^{2\alpha} + |y|^{2\alpha} - |x-y|^{2\alpha}) 
= c_\alpha \int_{\R} \left[\frac{(1-\cos{x\xi})(1-\cos{y\xi})}{|\xi|^{2\alpha+1}}
+\frac{\sin{x\xi}\sin{y\xi}}{|\xi|^{2\alpha+1}}\right]\,d\xi, \]
where $c_\alpha$ is a constant depending on $\alpha$.
This implies that for all $x, y \in \R^N$,
\begin{equation}\label{fBscov}
\prod_{j=1}^N \frac 1 2 (|x_j|^{2\alpha} + |y_j|^{2\alpha} - |x_j - y_j|^{2\alpha})
= c_\alpha^N \sum_{p \in \{0, 1\}^N} \int_{\R^N} \prod_{i=1}^N 
\frac{f_{p_i}(x_i\xi_i) f_{p_i}(y_i\xi_i)}{|\xi_i|^{2\alpha+1}} \, d\xi,
\end{equation}
where $f_0(x) = 1 - \cos{x}$ and $f_1(x) = \sin{x}$.
Therefore, $v$ has the following representation:
\begin{equation}\label{fBsHR}
v(x) \overset{d}{=} c \sum_{p \in \{0, 1\}^N} \int_{\R^N} \prod_{i=1}^N 
\frac{f_{p_i}(x_i \xi_i)}{|\xi_i|^{\alpha+1/2}} \,W_p(d\xi),
\end{equation}
where $c = c_\alpha^{N/2}$ and $W_p(d\xi)$, $p \in \{0, 1\}^N$, 
are i.i.d.\ $\R^d$-valued white noises on $\R^N$. 
In what follows, we assume that $v$ is defined by this representation,
which allows us to introduce a set variable $A \in \mathscr{B}(\R_+)$ 
and define
\[ v(A, x) = c \sum_{p \in \{0, 1\}^N} \int_{{|\xi|}_\infty \in A} \prod_{i=1}^N 
\frac{f_{p_i}(x_i \xi_i)}{|\xi_i|^{\alpha+1/2}} \,W_p(d\xi), \]
where ${|\xi|}_\infty = \max\{ |\xi_j| : 1 \le j \le N \}$.
For each $A$, $x \mapsto v(A, x)$ has a continuous version because of 
$\|v(A, x) - v(A, y)\|_{L^2} \le \|v(x) - v(y)\|_{L^2}$ and \eqref{fBs1},
and we will use such a version.

Let $T$ be a compact interval in $(0, \infty)^N$.
Fix $s \in T$.
We are going to verify Assumption \ref{a2}.
By re-centering and re-scaling, we can consider the intervals 
$\prod_{j=1}^N [s_j, s_j+r]$ instead of $\prod_{j=1}^N [s_j-r, s_j+r]$.
This will simplify notations.

Motivated by the representation \eqref{Bs1} for the Brownian sheet,
we define, for each $j$, the Gaussian process 
$\{ \tilde v^j(A, x_j) : A \in \mathscr{B}(\R_+), x_j \ge s_j \}$ (depending on $s$)
by
\begin{align*}
\tilde v^j(A, x_j) &:= v(A, (s_1, \dots, s_{j-1}, x_j, s_{j+1}, \dots, s_N)) - v(A, s)\\
& = c \sum_{p \in \{0, 1\}^N} 
\int_{{|\xi|}_\infty \in A} \frac{f_{p_j}(x_j\xi_j) - f_{p_j}(s_j\xi_j)}{|\xi_j|^{\alpha+1/2}} 
\prod_{i\ne j} \frac{f_{p_i}(s_i \xi_i)}{|\xi_i|^{\alpha+1/2}} \,W_p(d\xi).
\end{align*}
Set $\tilde v(A, x) = \sum_{j=1}^N \tilde v^j(A, x_j)$.

\begin{lemma}\label{fBsa2}
Fix $s \in T$. Take $r_0 = 1$, and let $\tilde v^j(A, x_j)$ and $\tilde v(A, x)$ 
be defined as above.
Then there exists a finite constant $c$ depending on $T$ but not on $s$ 
such that the following statements hold.
\begin{enumerate}
\item[(a)] For each $j$ and $x_j$, $A \mapsto \tilde v^j(A, x_j)$ is an $\R^d$-valued 
independently scattered Gaussian measure.
The $\sigma$-algebra
$\sigma\{ \tilde v^1(A, \cdot), \dots, \tilde v^N(A, \cdot) \}$ is independent of
$\sigma\{\tilde v^1(B, \cdot), \dots, \tilde v^N(B, \cdot) \}$ whenever 
$A$ and $B$ are fixed, disjoint sets.
\item[(b)] For all $j$, for all $x_j, y_j \in [s_j, s_j + r_0]$,
\[ {\|\tilde v^j(\R_+, x_j) - \tilde v^j(\R_+, y_j)\|}_{L^2} \le c\, |x_j - y_j|^\alpha. \]
\item[(c)] Take $\gamma_1 = 0$, $\gamma_2 = \alpha$ and $a_0 = 0$.
Then for all $0 < r \le r_0$, for all $x, y \in \prod_{j=1}^N[s_j, s_j+r]$, for 
all $0 \le a < b \le \infty$,
\begin{align}
\qquad {\|v(x) - v(y) - \tilde v([a, b), x) + \tilde v([a, b), y)\|}_{L^2} 
\le c \left( a^{1-\alpha}|x-y| + b^{-\alpha} + r^\alpha |x-y|^\alpha \right).
\end{align}
\end{enumerate}
\end{lemma}

\begin{proof}
First, (a) is obvious because $W_p$, $p \in \{0, 1\}^N$, are independent white noises,
and (b) follows from \eqref{fBs1} since 
\[\tilde v^j(\R_+, x_j) - \tilde v^j(\R_+, y_j) 
= v(s_1, \dots, s_{j-1}, x_j, s_{j+1}, \dots, s_N) - 
v(s_1, \dots, s_{j-1}, y_j, s_{j+1}, \dots, s_N).\]
It remains to prove (c). Let $x, y \in \prod_{j=1}^N[s_j, s_j + r]$. 
Notice the telescoping sum:
\begin{align*}
 \prod_{i=1}^N \frac{f_{p_i}(x_i\xi_i)}{|\xi_i|^{\alpha+1/2}} - 
\prod_{i=1}^N\frac{f_{p_i}(y_i\xi_i)}{|\xi_i|^{\alpha+1/2}}
 = \sum_{j=1}^N \Bigg(\frac{f_{p_j}(x_j\xi_j) - f_{p_j}(y_j\xi_j)}{|\xi_j|^{\alpha+1/2}} 
\prod_{i<j} \frac{f_{p_i}(y_i\xi_i)}{|\xi_i|^{\alpha+1/2}} 
\prod_{i>j} \frac{f_{p_i}(x_i\xi_i)}{|\xi_i|^{\alpha+1/2}}\Bigg).
\end{align*}
It follows that
\begin{align*}
&v(x) - v(y) - \tilde v([a, b), x) + \tilde v([a, b), y)\\
& = c \sum_{p \in \{0, 1\}^N} \sum_{j=1}^N \int_{\xi \in \R^N} 
\Bigg(\frac{f_{p_j}(x_j\xi_j) - f_{p_j}(y_j\xi_j)}{|\xi_j|^{\alpha+1/2}} 
\prod_{ i < j} \frac{f_{p_i}(y_i\xi_i)}{|\xi_i|^{\alpha+1/2}} 
\prod_{i>j} \frac{f_{p_i}(x_i\xi_i)}{|\xi_i|^{\alpha+1/2}}\Bigg) W_p(d\xi)\\
& \quad + c \sum_{p\in \{0, 1\}^N} \sum_{j=1}^N \int_{{|\xi|}_\infty \in [a, b)} 
\Bigg(\frac{f_{p_j}(x_j\xi_j) - f_{p_j}(y_j\xi_j)}{|\xi_j|^{\alpha+1/2}} 
\prod_{i \ne j} \frac{f_{p_i}(s_i\xi_i)}{|\xi_i|^{\alpha+1/2}} \Bigg) W_p(d\xi)\\
& = c \sum_{p\in \{0, 1\}^N} 
\sum_{j=1}^N \Bigg[\int_{{|\xi|}_\infty \in \R_+ \setminus [a, b)}
\Bigg(\frac{f_{p_j}(x_j\xi_j) - f_{p_j}(y_j\xi_j)}{|\xi_j|^{\alpha+1/2}} 
\prod_{i < j} \frac{f_{p_i}(y_i\xi_i)}{|\xi_i|^{\alpha+1/2}} 
\prod_{i > j} \frac{f_{p_i}(x_i\xi_i)}{|\xi_i|^{\alpha+1/2}}\Bigg) W_p(d\xi)\\
& \quad +  \int_{{|\xi|}_\infty \in [a, b)}
\frac{f_{p_j}(x_j\xi_j) - f_{p_j}(y_j\xi_j)}{|\xi_j|^{\alpha+1/2}} 
\Bigg( \prod_{i < j} \frac{f_{p_i}(y_i\xi_i)}{|\xi_i|^{\alpha+1/2}} 
\prod_{i>j} \frac{f_{p_i}(x_i\xi_i)}{|\xi_i|^{\alpha+1/2}}
- \prod_{i \ne j} \frac{f_{p_i}(s_i\xi_i)}{|\xi_i|^{\alpha+1/2}}\Bigg) W_p(d\xi) \Bigg]\\
& =: c \sum_{p\in \{0, 1\}^N} 
\sum_{j=1}^N \big[u_{p, j}^1(x, y) + u_{p, j}^2(x, y)\big].
\end{align*}
Fix $p \in \{0, 1\}^N$ and $1 \le j \le N$.
Let us consider $u_{p, j}^1$ first. Note that
\[ \big\{ {|\xi|}_\infty \in \R_+ \setminus [a, b) \big\} 
= \big\{ |\xi_k| < a, \forall\, 1 \le k \le N \big\} \cup 
\bigcup_{k=1}^N \big\{ |\xi_k| \ge b \big\}. \]
Also, $|f_{p_j}(x_j \xi_j) - f_{p_j}(y_j\xi_j)| \le |x_j - y_j| |\xi_j|$ and
$|f_{p_j}(x_j \xi_j) - f_{p_j}(y_j\xi_j)| \le 2$ for $p_j = 0, 1$.
It follows that
\begin{align*}
&{\|u_{p, j}^1(x, y)\|}_{L^2}^2\\
& \le d \int_{\{|\xi_k| < a, \forall k\}} 
\frac{|x_j - y_j|^2}{|\xi_j|^{2\alpha-1}}
\Bigg(\prod_{i < j} \frac{f_{p_i}(y_i\xi_i)}{|\xi_i|^{\alpha+1/2}} 
\prod_{i > j} \frac{f_{p_i}(x_i\xi_i)}{|\xi_i|^{\alpha+1/2}}\Bigg)^2 d\xi\\
& \quad + d \int_{\{|\xi_j| \ge b\}} \frac{4}{|\xi_j|^{2\alpha+1}}
\Bigg(\prod_{i < j} \frac{f_{p_i}(y_i\xi_i)}{|\xi_i|^{\alpha+1/2}} 
\prod_{i > j} \frac{f_{p_i}(x_i\xi_i)}{|\xi_i|^{\alpha+1/2}}\Bigg)^2 d\xi\\
& \quad +d \sum_{k\ne j} \int_{\{ |\xi_k| \ge b\}} 
\frac{1}{|\xi_k|^{2\alpha+1}}
\Bigg(\frac{f_{p_j}(x_j\xi_j) - f_{p_j}(y_j\xi_j)}{|\xi_j|^{\alpha+1/2}}
\prod_{i < j, i\ne k} \frac{f_{p_i}(y_i\xi_i)}{|\xi_i|^{\alpha+1/2}} 
\prod_{i > j, i\ne k} \frac{f_{p_i}(x_i\xi_i)}{|\xi_i|^{\alpha+1/2}}\Bigg)^2 d\xi.
\end{align*}
By \eqref{fBscov} applied to $\R^{N-1}$ instead of $\R^N$, we have
\[ \int_{\R^{N-1}} \Bigg(\prod_{i < j} \frac{f_{p_i}(y_i\xi_i)}{|\xi_i|^{\alpha+1/2}} 
\prod_{i > j} \frac{f_{p_i}(x_i\xi_i)}{|\xi_i|^{\alpha+1/2}}\Bigg)^2 d\xi
\le C \prod_{i < j} |y_i|^{2\alpha} \prod_{i>j} |x_i|^{2\alpha}.\]
Also, using the representation \eqref{fBsHR} and the estimate
\eqref{fBs1} for a fractional Brownian sheet on $\R^{N-1}_+$, we get that
\[ \int_{\R^{N-1}} \Bigg(\frac{f_{p_j}(x_j\xi_j) - f_{p_j}(y_j\xi_j)}{|\xi_j|^{\alpha+1/2}}
\prod_{i < j, i\ne k} \frac{f_{p_i}(y_i\xi_i)}{|\xi_i|^{\alpha+1/2}} 
\prod_{i > j, i\ne k} \frac{f_{p_i}(x_i\xi_i)}{|\xi_i|^{\alpha+1/2}}\Bigg)^2 d\xi 
\le C |x_j - y_j|^{2\alpha}. \]
Then
\begin{align*}
{\|u_{p, j}^1(x, y)\|}_{L^2}^2 
&\le C \bigg( a^{2 - 2\alpha} |x_j - y_j|^2 
\prod_{i < j} |y_i|^{2\alpha} \prod_{i>j} |x_i|^{2\alpha}\\
& \qquad + b^{-2\alpha} \prod_{i<j}|y_i|^{2\alpha} \prod_{i>j}|x_i|^{2\alpha}
+ b^{-2\alpha} |x_j-y_j|^{2\alpha}  \bigg)\\
& \le C\left( a^{2 - 2\alpha} |x-y|^2 + b^{-2\alpha} \right).
\end{align*}
For $u^2_{p, j}$, we use \eqref{fBs1} in a similar way to deduce that
\begin{align*}
{\|u^2_{p, j}(x, y)\|}_{L^2}^2 & \le 
d\int_{\R^N} \Bigg(\frac{f_{p_j}(x_j\xi_j) - f_{p_j}(y_j\xi_j)}{|\xi_j|^{\alpha+1/2}} \Bigg)^2
\Bigg( \prod_{i < j} \frac{f_{p_i}(y_i\xi_i)}{|\xi_i|^{\alpha+1/2}} 
\prod_{i>j} \frac{f_{p_i}(x_i\xi_i)}{|\xi_i|^{\alpha+1/2}}
- \prod_{i \ne j} \frac{f_{p_i}(s_i\xi_i)}{|\xi_i|^{\alpha+1/2}}\Bigg)^2 d\xi\\
& \le C \Bigg[ |x_j - y_j|^{2\alpha} 
\bigg(\sum_{i<j} |y_i - s_i|^{2\alpha} + \sum_{i > j} |x_i - s_i|^{2\alpha}\bigg) \Bigg]\\
& \le C |x-y|^{2\alpha}r^{2\alpha}.
\end{align*}
Therefore, we have
\[ {\|v(x) - v(y) - \tilde v([a, b), x) + \tilde v([a, b), y)\|}_{L^2}^2 
\le C(a^{2 - 2\alpha} |x-y|^2 + b^{-2\alpha} + r^{2\alpha} |x-y|^{2\alpha}). \]
This proves (c) and completes the lemma.
\end{proof}

It is known that when $N > \alpha d$, the fractional Brownian sheet $v$ 
has a jointly continuous local time $L(z, J)$ on any closed interval $J$ in $(0, \infty)^N$.
See \cite{XZ02, AWX08}.

\begin{theorem}
Let $T$ be a compact interval in $(0, \infty)^N$.
\begin{enumerate}
\item[(a)] If $N < \alpha d$, then there exist positive finite constants $C_1$ and $C_2$ 
such that 
\[ \P\Big\{ C_1 \lambda(J) \le \mathscr{H}_\phi(v(J)) \le C_2 \lambda(J) \text{ for all }
J \in \mathscr{I}(T) \Big\} = 1, \]
where $\phi(r) = r^{N/\alpha}(\log\log(1/r))^N$.
\item[(b)] If $N > \alpha d$, then there exists a positive finite constant $C$ such that
for any $z \in \R^d$,
\[ \P\Big\{ C L(z, J) \le \mathscr{H}_\varphi(v^{-1}(z) \cap J) < \infty \text{ for all }
J \in \mathscr{I}(T) \Big\} = 1, \]
where $\varphi(r) = r^{N - \alpha d}(\log\log(1/r))^{\alpha d}$.
If $N \le \alpha d$, then $v^{-1}(z) \cap T = \varnothing$ a.s.
\end{enumerate}
\end{theorem}

\begin{proof}
By \eqref{fBs1} and \eqref{fBsLND} above, $v$ satisfies Assumption \ref{a1}.
By Lemma \ref{fBsa2}, it satisfies Assumption \ref{a2}.
Moreover, Lemma 5.2 of \cite{DLMX} shows that Assumption \ref{a3} holds.
Therefore, the results follow from Theorems \ref{thm1} and \ref{thm2}.
\end{proof}

Theorem \ref{fBsLIL} below is a Chung-type law of the iterated logarithm for the 
fractional Brownian sheet, which is a direct consequence of Theorem \ref{thm:LIL}. 
We point out that the exponent $-\alpha$ of the logarithmic factor is different 
compared to the fractional Brownian motion from $\R^N$ to $\R^d$, 
for which the exponent is $-\alpha/N$ (cf. \cite{LS}).
Theorem \ref{fBsLIL} only holds for intervals away from the origin and the axes. 
In fact, Chung's law at the origin takes a different form for the Brownian sheet
\cite{T93} (see also \cite{LS}), and this is open for fractional Brownian sheets.

\begin{theorem}\label{fBsLIL}
Let $T$ be a compact interval in $(0, \infty)^N$. Then for any fixed $x_0 \in T$,
there exists a constant $K$ depending on $x_0$ such that
\[ \liminf_{r\to 0} \sup_{x:{|x-x_0|}_{\infty} \le r} 
\frac{|v(x) - v(x_0)|}{r^{\alpha}(\log\log(1/r))^{-\alpha}} = K \quad \text{a.s.} \]
and $K_1 \le K \le K_2$ for some positive finite constants $K_1$ and $K_2$ that
depend on $T$.
\end{theorem}

\medskip

\subsection{Systems of stochastic wave equations}

Consider the following system of stochastic wave equations for $t \ge 0$, $x \in \R$:
\begin{equation}\label{SWE}
\begin{cases}
\displaystyle \frac{\partial^2}{\partial t^2} U_j(t, x) = \Delta U_j(t, x) + \tilde{W}_j(t, x), 
\quad j = 1, \dots, d,\\
\displaystyle U_j(0, x) = 0, \quad \frac{\partial}{\partial t} U_j(0, x) = 0.
\end{cases}
\end{equation}
Here, $\tilde{W} = (\tilde{W}_1, \dots, \tilde{W}_d)$ is an $\R^d$-valued Gaussian noise.
We assume that $\tilde{W}_1, \dots, \tilde{W}_d$ are i.i.d.\ and 
each $\tilde{W}_j$ is either
\begin{enumerate}
\item[(i)] white in time and colored in space with covariance:
\[ \E[\tilde{W}(t, x) \tilde{W}(s, y)] = \delta_0(t-s) |x-y|^{-\beta}, \]
where $0 < \beta < 1$, or
\item[(ii)] a space-time white noise (set $\beta = 1$ in this case).
\end{enumerate}
The noise $\tilde W_j$ is defined as a centered Gaussian process $\tilde W_j(\phi)$
indexed by compactly supported smooth functions $\phi \in C^\infty_c(\R_+ \times \R)$
such that for all $\phi_1, \phi_2 \in C_c^\infty(\R_+ \times \R)$, 
\begin{equation*}
\E[\tilde W_j(\phi_1) \tilde W_j(\phi_2)] 
= \int_{\R_+} ds \int_\R dy \int_\R dy' \, \phi_1(s, y) f(y, y') \phi_2(s, y'),
\end{equation*}
where $f(y, y') = |y-y'|^{-\beta}$ in case (i), and $f(y, y') = \delta_0(y - y')$ in case (ii).
Following \cite{DF98, D99}, $\tilde W_j$ extends to 
a $\sigma$-finite $L^2$-valued measure $\tilde W_j(A)$, 
for bounded Borel sets $A \in \mathscr{B}_b(\R_+\times \R)$, such that
\begin{equation}\label{Wcov}
\begin{split}
\E[\tilde W_j(A) \tilde W_j(B)]
& = \int_{\R_+} ds \int_\R dy \int_\R dy'\, {\bf 1}_A(s, y) f(y, y') {\bf 1}_B(s, y')\\
& = \frac{C_\beta}{2\pi}\int_{\R_+} ds \int_\R \,d\xi\,
\mathscr{F}{\bf 1}_A(s, \cdot)(\xi) \overline{\mathscr{F}{\bf 1}_B(s, \cdot)(\xi)}\,
|\xi|^{\beta-1},
\end{split}
\end{equation}
where $\mathscr{F}\phi(s, \cdot)$ is the Fourier transform of $y \mapsto \phi(s, y)$ 
defined by $\mathscr{F}\phi(s, \cdot)(\xi) = \int_\R e^{-iy\xi} \phi(s, y) dy$.
The solution of \eqref{SWE} is given by
\begin{equation}\label{U}
U(t, x) = \frac 1 2 \int_0^t \int_{\R} {\bf 1}_{\{ |x-y| \le t-s \}} \tilde W(ds,dy)
= \frac 1 2 \tilde W(\Delta(t, x)),
\end{equation}
where $\Delta(t, x) = \{ (s, y) : 0 \le s \le t, |x - y| \le t-s \}$,
and $U = \{ U(t, x) : t \ge 0, x \in \R \}$ is a Gaussian field with
i.i.d.\ components $U_1, \dots, U_d$.
When $\tilde W$ is a space-time white noise, it is known that
\begin{equation}\label{mBs}
U(t, x) = \frac 1 2 \hat W\Big(\frac{t-x}{\sqrt 2}, \frac{t+x}{\sqrt 2}\Big),
\end{equation}
where $\hat W$ is the modified Brownian sheet defined by Walsh \cite[Theorem 3.1]{W}.

By Proposition 4.1 of \cite{DS10}, for any compact interval 
$T \subset (0, \infty) \times \R$, there exists a positive finite constant $c_0$
such that for all $(t, x), (s, y) \in T$,
\begin{equation}\label{SWE1}
c_0^{-1} (|t-s| + |x-y|)^{2-\beta} \le \E(|U(t, x) - U(s, y)|^2) \le c_0 (|t-s|+|x-y|)^{2-\beta},
\end{equation}
and by Proposition 2.1 of \cite{LX19}, $U$ satisfies sectorial LND: 
there exist constants $c_1 > 0$ and 
$\delta_0 > 0$ such that for all $n \ge 1$,
for all $(t, x), (t^1, x^1), \dots, (t^n, x^n) \in T$ with 
$|t-t^i| + |x-x^i| \le \delta_0$ for all $i$,
\begin{equation}\label{SWELND}
\begin{split}
&\mathrm{Var}(U_1(t, x)| U_1(t^1, x^1), \dots, U_1(t^n, x^n))\\
&\ge c_1 \,\Big(\min_{1 \le i \le n} |(t+x)-(t^i+x^i)|^{2-\beta} +
\min_{1 \le i \le n} |(t-x)-(t^i-x^i)|^{2-\beta}\Big).
\end{split}
\end{equation}

By Proposition 9.2 of \cite{DMX17}, the solution $U$ of \eqref{SWE} has the same law
as the Gaussian random field $V = \{ V(t, x) : t \ge 0, x \in \R \}$ defined by
\begin{equation}\label{SWEHR}
V(t, x) = C_0\, \Re \int_\R \int_\R F(t, x, \tau, \xi) |\xi|^{(\beta-1)/2}\, W(d\tau, d\xi),
\end{equation}
where $C_0$ is a constant,
\begin{equation*}
F(t, x, \tau, \xi) = \frac{e^{-ix\xi}}{2|\xi|} 
\Bigg( \frac{e^{-it\tau} - e^{it|\xi|}}{\tau + |\xi|} - \frac{e^{-it\tau} - e^{-it|\xi|}}{\tau - |\xi|}
\Bigg)
\end{equation*}
and $W$ is a $\C^d$-valued space-time white noise, that is, 
$\Re W$ and $\Im W$ are independent $\R^d$-valued space-time white noises
with i.i.d.\ components. 
Here, $\Re$ and $\Im$ stand for real part and imaginary part respectively.
We refer to \eqref{SWEHR} as the \emph{harmonizable representation} for the solution
$U(t, x)$ of the stochastic wave equation.

In view of \eqref{mBs} and \eqref{SWELND}, 
it is natural to change coordinates
by a rotation of $45^\circ$:
\[ (\eta, \theta) = \Big(\frac{t-x}{\sqrt{2}}, \frac{t+x}{\sqrt 2}\Big), \quad \text{or} \quad
(t, x) = \Big(\frac{\eta+\theta}{\sqrt 2}, \frac{-\eta+\theta}{\sqrt 2}\Big). \]
With the $(\eta, \theta)$ coordinate system, we write $f(\eta, \theta, \tau, \xi) = 
F(\frac{\eta+\theta}{\sqrt 2}, \frac{-\eta+\theta}{\sqrt 2}, \tau, \xi)$ and
\begin{equation}\label{HRv}
v(\eta, \theta) = 
C_0\, \Re \int_\R \int_\R f(\eta, \theta, \tau, \xi) |\xi|^{(\beta-1)/2}\, W(d\tau, d\xi).
\end{equation}
It follows that, as processes on $\{ (\eta, \theta) : \eta + \theta \ge 0 \}$,
\begin{equation}\label{SWEd}
v(\eta, \theta) \overset{d}{=} 
U\Big( \frac{\eta+\theta}{\sqrt 2}, \frac{-\eta+\theta}{\sqrt 2} \Big).
\end{equation}
For any $A \in \mathscr{B}(\R_+)$, define
\[ v(A, \eta, \theta) = C_0 \, \Re \iint_{|\tau| \vee |\xi| \in A} 
f(\eta, \theta, \tau, \xi) |\xi|^{(\beta-1)/2} \,W(d\tau, d\xi), \]
where $|\tau|\vee|\xi| = \max\{ |\tau|, |\xi|\}$.
For each $A$, by $\|v(A, \eta, \theta) - v(A, \eta', \theta')\|_{L^2} 
\le \|v(\eta, \theta) - v(\eta', \theta')\|_{L^2}$ and \eqref{SWE1}, 
we can choose a version of $v(A, \cdot)$ such that 
$(\eta, \theta)\mapsto v(A, \eta, \theta)$ is 
continuous.

Let $T$ be a compact rectangle in $\{ (\eta, \theta) : \eta + \theta > 0 \}$.
Suppose $(\eta_0, \theta_0) \in T$ is fixed.
Then for any $\eta \ge \eta_0$ and $\theta \ge \theta_0$, define
\begin{align*}
\tilde v^1(A, \eta) &= v(A, \eta, \theta_0) - v(A, \eta_0, \theta_0),\\
\tilde v^2(A, \theta) &= v(A, \eta_0, \theta) - v(A, \eta_0, \theta_0).
\end{align*}
Set $\tilde v(A, \eta, \theta) = \tilde v^1(A, \eta) + \tilde v^2(A, \theta)$

\begin{lemma}\label{SWEa2}
Let $T$ be a compact rectangle in $\{ (\eta, \theta) : \eta + \theta > 0 \}$ and
let $(\eta_0, \theta_0) \in T$ be fixed. 
Take $r_0= 1$ and $\alpha = (2-\beta)/2$. Then, for some constants $c$
and $a_0 >1$, 
the following statements hold.
\begin{enumerate}
\item[(a)] For each $\eta \in [\eta_0, \eta_0 + r_0]$ and 
$\theta \in [\theta_0, \theta_0 + r_0]$,
$A \mapsto \tilde v^1(A, \eta)$ and $A \mapsto \tilde v^2(A, \theta)$
are independently scattered Gaussian measures.
Also, the $\sigma$-algebra $\sigma\{ \tilde v^1(A, \cdot), \tilde v^2(A, \cdot) \}$ 
is independent of $\sigma\{ \tilde v^1(B, \cdot), \tilde v^2(B, \cdot) \}$ whenever 
$A$ and $B$ are fixed, disjoint sets.
\item[(b)] For all $\eta_1, \eta_2 \in [\eta_0, \eta_0 + r_0]$,
\[ {\|\tilde v^1(\R_+, \eta_1) - \tilde v^1(\R_+, \eta_2)\|}_{L^2} 
\le c \, |\eta_1 - \eta_2|^\alpha. \]
For all $\theta_1, \theta_2 \in [\theta_0, \theta_0 + r_0]$, 
\[ {\|\tilde v^2(\R_+, \theta_1) - \tilde v^2(\R_+, \theta_2)\|}_{L^2} 
\le c \, |\theta_1 - \theta_2|^\alpha. \]
\item[(c)] 
Take $\gamma_1 = 0$ if $\beta = 1$; 
take $\gamma_1$ such that $\frac{1-\beta}{2} < \gamma_1< \frac 1 2$ if $0 < \beta < 1$.
Take $\gamma_2 = 1/2$.
Then for all $0 < r \le r_0$, for all
$(\eta_1, \theta_1), (\eta_2, \theta_2) \in [\eta_0, \eta_0 + r]\times [\theta_0, \theta_0 + r]$,
and $a_0 \le a < b \le \infty$,
\begin{align*}
&{\|v(\eta_1, \theta_1) - v(\eta_2, \theta_2) - \tilde v([a, b), \eta_1, \theta_1)
+ \tilde v([a, b), \eta_2, \theta_2)\|}_{L^2}\\
&\le c\, \Big[ a^{1-\alpha} \big(|\eta_1 - \eta_2| + |\theta_1 - \theta_2|\big) 
+ r^{\gamma_1} b^{\gamma_1 -\alpha} 
+ r^{1/2} \big(|\eta_1 - \eta_2| + |\theta_1 - \theta_2|\big)^\alpha \Big].
\end{align*}
\end{enumerate}
\end{lemma}

\begin{proof}
(a) is obvious because $W$ is a space-time white noise, and 
(b) follows from \eqref{SWE1}.
It remains to prove (c).
Let $0 < r \le r_0$ and 
$(\eta_1, \theta_1), (\eta_2, \theta_2) \in [\eta_0, \eta_0 + r]\times [\theta_0, \theta_0 + r]$.
For simplicity, we will suppress the $\tau, \xi$ variables in $f(\eta, \theta, \tau, \xi)$.
Note that
\begin{align*}
&v(\eta_1, \theta_1) - v(\eta_2, \theta_2) - \tilde v([a, b), \eta_1, \theta_1)
+ \tilde v([a, b), \eta_2, \theta_2)\\
& = C_0\, \Re \iint_{\R^2} \big[ f(\eta_1, \theta_1) - f(\eta_2, \theta_2)
\big] |\xi|^{(\beta-1)/2}\, W(d\tau,d\xi)\\
& \quad - C_0\, \Re \iint_{|\tau|\vee|\xi| \in [a, b)} \big[f(\eta_1, \theta_0)
- f(\eta_2, \theta_0)\big] |\xi|^{(\beta-1)/2}\, W(d\tau,d\xi)\\
& \quad - C_0\, \Re \iint_{|\tau|\vee|\xi| \in [a, b)} \big[f(\eta_0, \theta_1)
- f(\eta_0, \theta_2)\big] |\xi|^{(\beta-1)/2}\, W(d\tau,d\xi).
\end{align*}
This is equal to
\begin{align*}
& C_0\bigg( \Re \iint_{|\tau|\vee|\xi| < a} 
\big[ f(\eta_1, \theta_1) - f(\eta_2, \theta_2)\big] 
|\xi|^{(\beta-1)/2}\, W(d\tau,d\xi)\\
& \quad + \Re \iint_{|\tau|\vee|\xi| \ge b} 
\big[ f(\eta_1, \theta_1) - f(\eta_2, \theta_2)\big] 
|\xi|^{(\beta-1)/2}\, W(d\tau,d\xi)\\
& \quad - \Re \iint_{|\tau|\vee|\xi| \in [a, b)} 
\big[f(\eta_1, \theta_0) - f(\eta_2, \theta_0) - f(\eta_1, \theta_1) + f(\eta_2, \theta_1)\big] 
|\xi|^{(\beta-1)/2}\, W(d\tau,d\xi)\\
& \quad - \Re \iint_{|\tau|\vee|\xi| \in [a, b)} 
\big[f(\eta_0, \theta_1) - f(\eta_0, \theta_2) - f(\eta_2, \theta_1) + f(\eta_2, \theta_2)\big] 
|\xi|^{(\beta-1)/2}\, W(d\tau,d\xi) \bigg)\\
& =: C_0 (w_1 + w_2 - w_3 -  w_4).
\end{align*}
Switching back to the $(t, x)$ coordinates, let 
$(t_i, x_i) = (\frac{\eta_i+\theta_i}{\sqrt 2}, \frac{-\eta_i+\theta_i}{\sqrt 2})$ for $i = 1, 2$.
By Lemmas 9.4(a) and 9.5(a) of \cite{DMX17}, we get that
${\|w_1\|}^2_{L^2}
\le C (|t_1 - t_2| + |x_1 - x_2|)^2\, a^{\beta}
\le C (|\eta_1 - \eta_2| + |\theta_1 - \theta_2|)^2\, a^{\beta}$.
Notice that $\beta = 2-2\alpha$.

Now consider $w_2$.
If $\beta = 1$, then by Lemmas 9.4(b) and 9.5(b) of \cite{DMX17}, we have
${\|w_2\|}^2_{L^2} \le C b^{-2\alpha}$ 
for $b$ large enough, in particular when $a_0$ is large enough.
Suppose $0 < \beta < 1$. 
Write
\[ w_2 = \Re \iint_{|\tau|\vee|\xi| \ge b}[F(t_1, x_1) - F(t_1, x_2) - F(t_1, x_2) + F(t_2, x_2)]
|\xi|^{(\beta-1)/2} W(d\tau,d\xi). \]
Take $\gamma_1$ such that $1-\beta < 2\gamma_1 < 1$. 
We claim that
\begin{align*}
&I_1 := \iint_{|\tau|\vee |\xi| \ge b} 
|F(t_1, x_1, \tau, \xi) - F(t_1, x_2, \tau, \xi)|^2 |\xi|^{\beta-1} d\tau d\xi 
\le C r^{2\gamma_1} b^{2\gamma_1-2+\beta},\\
&I_2 := \iint_{|\tau|\vee |\xi| \ge b} 
|F(t_1, x_2, \tau, \xi) - F(t_2, x_2, \tau, \xi)|^2 |\xi|^{\beta-1} d\tau d\xi 
\le C r^{2\gamma_1} b^{2\gamma_1-2+\beta}.
\end{align*}
Consider $I_1$ first:
\[ I_1 = \iint_{|\tau|\vee|\xi| \ge b} \frac{|1 - e^{-i(x_1-x_2)\xi}|^2}{|\xi|^2}\,
\bigg| \frac{1-e^{it_1(\tau+|\xi|)}}{\tau+|\xi|} - \frac{1-e^{it_1(\tau-|\xi|)}}{\tau-|\xi|} \bigg|^2 
|\xi|^{\beta-1}d\tau d\xi. \]
Since $|1-e^{iz}| \le 2\wedge |z|$ and $|x_1-x_2|\le \sqrt 2 r$, we have 
$|1 - e^{-i(x_1-x_2)\xi}|^2 \le 2^{2-2\gamma_1} |\sqrt 2 r\xi|^{2\gamma_1}$. 
Moreover, according to Lemma 9.4(b) of \cite{DMX17},
\begin{equation}\label{Fbd}
\frac{1}{|\xi|}\bigg| \frac{1-e^{it(\tau+|\xi|)}}{\tau+|\xi|} - 
\frac{1-e^{it(\tau-|\xi|)}}{\tau-|\xi|} \bigg| \le 
C \left( \frac{1}{1+ |\tau + |\xi||} + \frac{1}{1 + |\tau - |\xi||} \right)
\frac{1}{1+|\xi|}.
\end{equation}
Then, by the elementary inequality $x^2 + y^2 \le (x+y)^2 \le 2(x^2+y^2)$ 
for $x, y \ge 0$, it follows that
\begin{align*}
I_1 \le C r^{2\gamma_1} \iint_{|\tau|\vee |\xi| \ge b}
\left( \frac{1}{1+(\tau+|\xi|)^2} + \frac{1}{1+(\tau-|\xi|)^2} \right) 
\frac{|\xi|^{2\gamma_1 + \beta-1}}{(1+|\xi|)^2} d\tau d\xi.
\end{align*}
Since the integrand is symmetric in $\xi$, 
it is enough to consider the integral for $\xi > 0$.
We need to integrate over the regions
\[ \{\xi \ge b, |\tau| \le \xi \} \text{ and } \{ 0 < \xi \le |\tau|, |\tau| \ge b \}. \]
For the first region, noting that $2\gamma_1 - 3 + \beta < -1$, we have
\begin{align*}
\int_b^\infty d\xi \frac{\xi^{2\gamma_1+\beta-1}}{(1+\xi)^2} 
\int_{-\xi}^\xi \frac{d\tau}{1+(\tau\pm \xi)^2} 
\le \int_b^\infty \xi^{2\gamma_1-3+\beta} d\xi
\int_{-\infty}^\infty \frac{d\tau}{1+\tau^2} 
\le C b^{2\gamma_1-2+\beta}.
\end{align*}
For the second region, we split it further into two parts 
(i) $\xi \ge b$, $|\tau| \ge \xi$ and (ii) $0 < \xi < b$, $|\tau| \ge b$. For part (i), we have
\begin{align*}
\int_b^\infty d\xi \frac{\xi^{2\gamma_1+\beta-1}}{(1+\xi)^2} 
\int_\xi^\infty \frac{d\tau}{1+(\tau\pm\xi)^2}
\le \int_b^\infty \xi^{2\gamma_1-3+\beta} d\xi
\int_{-\infty}^\infty \frac{d\tau}{1+\tau^2} 
\le C b^{2\gamma_1-2+\beta}.
\end{align*}
For part (ii), in case of a ``$+$'' sign in $\tau \pm \xi$,
\begin{align*}
\int_0^b d\xi \frac{\xi^{2\gamma_1+\beta-1}}{(1+\xi)^2} 
\int_b^\infty \frac{d\tau}{1+(\tau+\xi)^2}
\le \int_0^b \xi^{2\gamma_1-2+\beta} d\xi \int_b^\infty \frac{d\tau}{\tau^2}
\le C b^{2\gamma_1-2+\beta},
\end{align*}
where we have used the fact that $(1+\xi)^2 \ge \xi$ and $2\gamma_1 - 2 + \beta > -1$.
In case of a ``$-$'' sign, we consider $0 < \xi < b/2$ and $b/2 \le \xi < b$ respectively,
for which
\begin{align*}
\int_0^{b/2} d\xi \frac{\xi^{2\gamma_1+\beta-1}}{(1+\xi)^2} 
\int_b^\infty \frac{d\tau}{1+(\tau-\xi)^2}
\le \int_0^{b/2} \xi^{2\gamma_1-2+\beta} d\xi
\int_{b/2}^\infty \frac{d\tau}{\tau^2} 
\le C b^{2\gamma_1-2+\beta}
\end{align*}
and 
\begin{align*}
\int_{b/2}^b d\xi \frac{\xi^{2\gamma_1+\beta-1}}{(1+\xi)^2} 
\int_b^\infty \frac{d\tau}{1+(\tau-\xi)^2}
\le \int_{b/2}^b \xi^{2\gamma_1-3+\beta} d\xi
\int_{-\infty}^\infty \frac{d\tau}{1+\tau^2} 
\le C b^{2\gamma_1-2+\beta}.
\end{align*}
Hence $I_1 \le C r^{2\gamma_1}b^{2\gamma_1-2+\beta}$.

Now consider $I_2$. By Lemma 9.4(a) of \cite{DMX17},
\[ |F(t_1, x_2, \tau, \xi) - F(t_2, x_2, \tau, \xi)| 
\le C |t_1-t_2| \left( \frac{1}{1+ |\tau + |\xi||} + \frac{1}{1 + |\tau - |\xi||} \right). \]
This, together with \eqref{Fbd} and the elementary inequality 
$(x + y)^p \le 2^{p-1}(x^p + y^p)$
for $x, y \ge 0$ and $p = 2-2\gamma_1 > 1$, implies that
\begin{align*}
I_2 &\le \iint_{|\tau|\vee|\xi|\ge b} 
{\left(|F(t_1, x_2, \tau, \xi)| + |F(t_2, x_2, \tau, \xi)|\right)}^{2-2\gamma_1}
{|F(t_1, x_2, \tau, \xi) - F(t_2, x_2, \tau, \xi)|}^{2\gamma_1} |\xi|^{\beta-1} d\tau d\xi\\
& \le C |t_1-t_2|^{2\gamma_1} \iint_{|\tau|\vee|\xi|\ge b} 
\left( \frac{1}{1+|\tau+|\xi||} + \frac{1}{1+|\tau-|\xi||} \right)^2
\frac{|\xi|^{\beta-1}}{(1+|\xi|)^{2-2\gamma_1}} d\tau d\xi\\
& \le C r^{2\gamma_1} \iint_{|\tau|\vee|\xi|\ge b} 
\left( \frac{1}{1+(\tau+|\xi|)^2} + \frac{1}{1+(\tau-|\xi|)^2} \right)
\frac{|\xi|^{\beta-1}}{(1+|\xi|)^{2-2\gamma_1}} d\tau d\xi.
\end{align*}
Then, for the double integral, we can use similar calculations to those in the above to
show that 
$I_2 \le C r^{2\gamma_1} b^{2\gamma_1-2+\beta}$.
Therefore, we have
${\|w_2\|}_{L^2} \le C r^{\gamma_1} b^{\gamma_1-\alpha}$.

Finally, by Lemma \ref{lem:SWE} below, since 
$|\eta_1 - \eta_2| \le r$ and $|\theta_1 - \theta_0|\le r$, we get
${\|w_3\|}^2_{L^2} \le C r |\eta_2 - \eta_1|^{2-\beta}$.
Similarly, since $|\eta_2 - \eta_0| \le r$ and $|\theta_1 - \theta_2| \le r$, we have
${\|w_4\|}^2_{L^2} \le C r |\theta_2 - \theta_1|^{2-\beta}$.
This completes the proof of Lemma \ref{SWEa2}.
\end{proof}

\begin{lemma}\label{lem:SWE}
If $|\eta - \eta'| \le r$ and $|\theta - \theta'| \le r$, then
\begin{align}\label{SWEinc}
\begin{split}
&C_0^2 \iint_{\R^2} |f(\eta, \theta) - f(\eta', \theta) - f(\eta, \theta') + f(\eta', \theta')|^2 
|\xi|^{\beta-1} d\tau d\xi\\
& \hspace{150pt} \le C r \cdot
\min\{|\eta - \eta'|^{2-\beta}, |\theta - \theta'|^{2-\beta}\}. 
\end{split}
\end{align}
\end{lemma}

\begin{proof}
Without loss of generality, we may assume that $\eta' < \eta$ and $\theta' < \theta$.
Let $I$ denote the integral in \eqref{SWEinc}.
Then by \eqref{HRv},
\[ I = \E(|v(\eta, \theta) - v(\eta', \theta) - v(\eta, \theta') + v(\eta', \theta')|^2). 
\]
By \eqref{U} and \eqref{SWEd},
$I = \frac{d}{4}\, \E(|\tilde W_1(A)|^2)$,
where $A$ is the rotated rectangle with vertices 
\[\Big(\frac{\eta'+\theta'}{\sqrt 2}, \frac{-\eta'+\theta'}{\sqrt 2}\Big),
\Big(\frac{\eta+\theta'}{\sqrt 2}, \frac{-\eta+\theta'}{\sqrt 2}\Big),
\Big(\frac{\eta+\theta}{\sqrt 2}, \frac{-\eta+\theta}{\sqrt 2}\Big) \text{ and }
\Big(\frac{\eta'+\theta}{\sqrt 2}, \frac{-\eta'+\theta}{\sqrt 2}\Big). \]
Note that $A$ is contained in $A'$, where $A'$ is the parallelogram with vertices
\[\Big(\frac{\eta'+\theta'}{\sqrt 2}, \frac{-\eta'+\theta'}{\sqrt 2}\Big),
\Big(\frac{\eta'+\theta'}{\sqrt 2}, \frac{\eta'+\theta'-2\eta}{\sqrt 2}\Big),
\Big(\frac{\eta+\theta}{\sqrt 2}, \frac{-\eta+\theta}{\sqrt 2}\Big) \text{ and }
\Big(\frac{\eta+\theta}{\sqrt 2}, \frac{\eta+\theta - 2\eta'}{\sqrt 2}\Big). \]
Then by \eqref{Wcov}, since ${\bf 1}_A \le {\bf 1}_{A'}$ and $f \ge 0$, we have
\begin{align*}
\E(|\tilde W_1(A)|^2) & \le \E(|\tilde W_1(A')|^2)\\
& = \frac{C_\beta}{2\pi} \int_{\frac{\eta'+\theta'}{\sqrt 2}}^{\frac{\eta+\theta}{\sqrt 2}} ds 
\int_\R d\xi\,
|\mathscr{F}{\bf 1}_{[s - \sqrt 2 \eta, s - \sqrt 2 \eta']}(\xi)|^2\, |\xi|^{\beta-1}.
\end{align*}
For $a < b$, the Fourier transform of ${\bf 1}_{[a, b]}$ is
$\mathscr{F}{\bf 1}_{[a, b]}(\xi) = \frac{1}{i\xi}(e^{-ia\xi} - e^{-ib\xi})$.
Then by scaling, 
\[ \int_\R |\mathscr{F}{\bf 1}_{[a, b]}(\xi)|^2 \,|\xi|^{\beta-1} d\xi
= (b-a)^{2-\beta} \int_\R |1-e^{-i\xi}|^2 \,|\xi|^{\beta-3} d\xi. \]
Hence 
\[ \E(|\tilde W_1(A)|^2) \le C 
\int_{\frac{\eta'+\theta'}{\sqrt 2}}^{\frac{\eta+\theta}{\sqrt 2}} (\eta - \eta')^{2-\beta} ds
\le C r (\eta - \eta')^{2-\beta}.\]
Similarly, by considering $A \subset A''$, where $A''$ is the parallelogram with vertices
\[\Big(\frac{\eta'+\theta'}{\sqrt 2}, \frac{-\eta'+\theta'}{\sqrt 2}\Big),
\Big(\frac{\eta'+\theta'}{\sqrt 2}, \frac{2\theta-\eta'-\theta'}{\sqrt 2}\Big),
\Big(\frac{\eta+\theta}{\sqrt 2}, \frac{-\eta+\theta}{\sqrt 2}\Big) \text{ and }
\Big(\frac{\eta+\theta}{\sqrt 2}, \frac{2\theta'-\eta-\theta}{\sqrt 2}\Big), \]
we can show that 
$\E(|\tilde W_1(A)|^2) \le \E(|\tilde W_1(A'')|^2) \le C r (\theta - \theta')^{2-\beta}$.
The proof of \eqref{SWEinc} is complete.
\end{proof}

In \cite{DMX17}, it is shown that the stochastic wave equation
in spatial dimension $k \ge 1$ with $\beta \ge 1$ satisfies Assumption \ref{a3} with 
$\delta = 2-\beta$.
We verify this assumption for $k = 1$ and $0 < \beta \le 1$.

\begin{lemma}\label{SWEa3}
Let $T$ be a compact rectangle in $D := \{ (\eta, \theta) : \eta + \theta > 0 \}$.
Let $0 < \eps_0 \le 1$ be such that $T^{(\eps_0)} \subset\subset D$. 
Take $c = 1$ and $\delta = 1$.
\begin{enumerate}
\item[(a)] There exists a positive constant $c_3$ such that 
${\|v(\eta, \theta)\|}_{L^2} \ge c_3$ for all $(\eta, \theta) \in T^{(\eps_0)}$.
\item[(b)] Let $I \subset T$ be a compact rectangle and $0 < \rho \le \eps_0$.
For $(\eta_0, \theta_0) \in I$, take $(\eta_0', \theta_0') = (\eta_0, \theta_0)$.
Then there is a constant $c_4$ such that 
for all $(\eta_1, \theta_1), (\eta_2, \theta_2) \in I^{(\rho)}$ with 
$|\eta_1 - \eta| + |\theta_1 - \theta| \le 2\rho$ and 
$|\eta_2 - \eta| + |\theta_2 - \theta| \le 2\rho$,
for all $i = 1, \dots, d$,
\begin{equation}\label{SWEa3b} 
|\E[(v_i(\eta_1, \theta_1) - v_i(\eta_2, \theta_2)) v_i(\eta_0', \theta_0')]| 
\le c_4 (|\eta_1 - \eta_2| + |\theta_1 - \theta_2|).
\end{equation}
\end{enumerate}
\end{lemma}

\begin{proof}
For any $t > 0$, $x \in \R$, by \eqref{Wcov} and change of variables 
($r = t-s$ and then $\zeta = r \xi$), we have
\begin{align*}
{\|U(t, x)\|}_{L^2}^2 &= C \int_0^t ds \int_\R \frac{d\xi}{|\xi|^{1-\beta}}
\frac{\sin^2((t-s)|\xi|)}{|\xi|^2}\\
& = C \int_0^t dr\, r^{2-\beta} \int_\R \frac{d\zeta}{|\zeta|^{3-\beta}} \sin^2(|\zeta|)\\
& = C' t^{3-\beta}.
\end{align*}
Hence (a) is satisfied provided $t = \frac{\eta+\theta}{\sqrt 2} \ge c > 0$ for all 
$(\eta, \theta) \in T^{(\eps_0)}$, which is true since $T^{(\eps_0)} \subset \subset D$.
For (b), write $(t, x) = (\frac{\eta+\theta}{\sqrt 2}, \frac{-\eta+\theta}{\sqrt 2})$
and $(t_0', x_0') = (\frac{\eta_0'+\theta_0'}{\sqrt 2}, \frac{-\eta_0'+\theta_0'}{\sqrt 2})
= (\frac{\eta_0+\theta_0}{\sqrt 2}, \frac{-\eta_0+\theta_0}{\sqrt 2})$.
Then by \eqref{Wcov},
\begin{align*}
h(t, x)&:= \E[(U_i(t, x) U_i(t_0', x_0')]\\
&= C \int_0^{t\wedge t_0'} ds \int_\R \frac{d\xi}{|\xi|^{3-\beta}} 
e^{-i(x - x_0')\xi} \sin((t-s)|\xi|)\sin((t_0'-s)|\xi|).
\end{align*}
By considering $t < t_0'$ and $t \ge t_0'$ respectively, we see that
\begin{align*}
\frac{\partial}{\partial t} h(t, x) 
= C \int_0^{t\wedge t_0'} ds \int_\R \frac{d\xi}{|\xi|^{2-\beta}} e^{-i(x-x_0')\xi}
\cos((t-s)|\xi|) \sin((t_0'-s)|\xi|).
\end{align*}
Since $|\cos(z)| \le 1$ and $|\sin(z)|\le \min\{|z|, 1\}$,
\begin{align*}
\left|\frac{\partial}{\partial t} h(t, x)\right| 
&\le C \int_0^{t\wedge t_0'} ds \int_\R \frac{d\xi}{|\xi|^{2-\beta}} \min\{(t_0'-s)|\xi|, 1\}\\
& \le C \int_0^{t\wedge t_0'} ds 
\bigg( \int_{|\xi|\le (t_0'-s)^{-1}} (t'_0 - s)\frac{d\xi}{|\xi|^{1-\beta}} 
+ \int_{|\xi| > (t_0'-s)^{-1}} \frac{d\xi}{|\xi|^{2-\beta}}\bigg)\\
& \le C \int_0^{t\wedge t_0'} (t_0'-s)^{1-\beta} ds\\
& \le C (t\wedge t_0') (t_0')^{1-\beta},
\end{align*}
which is bounded over all possible $(\eta, \theta) \in I^{(\rho)}$ and 
$(\eta_0, \theta_0) \in I$ by the compactness of $I$.
For the derivative in $x$, we have
\begin{align*}
\frac{\partial}{\partial x} h(t, x) 
= C \int_0^{t\wedge t_0'} ds \int_\R \frac{d\xi}{|\xi|^{3-\beta}} i\xi e^{-i(x-x_0')\xi} 
\sin((t-s)|\xi|) \sin((t_0'-s)|\xi|).
\end{align*}
Then
\begin{align*}
\left|\frac{\partial}{\partial x} h(t, x)\right| 
&\le C \int_0^{t\wedge t_0'} ds \int_\R \frac{d\xi}{|\xi|^{2-\beta}} \min\{(t_0'-s)|\xi|, 1\},
\end{align*}
which is also bounded as shown by the calculations above. 
Therefore, by the mean value theorem, there is a finite constant $C$ such that for all 
$(\eta_1, \theta_1), (\eta_2, \theta_2) \in I^{(\rho)}$,
\[ |h(t_1, x_1) - h(t_2, x_2)| \le C (|t_1 - t_2| + |x_1 - x_2|), \]
which implies \eqref{SWEa3b}. This completes the proof of Lemma \ref{SWEa3}.
\end{proof}

By \eqref{SWE1} and \eqref{SWELND}, $v(\eta, \theta)$ satisfies Assumption \ref{a1}.
It follows from the results of \cite{X, WX11} that if $d < 4/(2-\beta)$, 
then $v(\eta, \theta)$ [and hence the solution $U(t, x)$ of \eqref{SWE}]
has a jointly continuous local time $L(z, J)$ on any compact rectangle 
$J$ away from $\{t = 0\}$.

We obtain the following result for the range and level sets of $v(\eta, \theta)$ 
in joint variables $(\eta, \theta)$.

\begin{theorem}
Suppose that $T$ is a compact rectangle in $\{(\eta, \theta) \in \R^2 : \eta + \theta > 0\}$.
Let $\phi(r) = r^{4/(2-\beta)}(\log\log(1/r))^2$ and 
$\varphi(r) = r^{2-d(2-\beta)/2}(\log\log(1/r))^{d(2-\beta)/2}$.
\begin{enumerate}
\item[(a)] If $d > 4/(2-\beta)$, then
there exist positive finite constants $C_1$ and $C_2$ such that 
\[ \P\Big\{ C_1 \lambda(J) \le \mathscr{H}_\phi(v(J)) \le C_2 \lambda(J) \text{ for all }
J \in \mathscr{I}(T) \Big\} = 1. \]
In particular, $\dim v(T) = 4/(2-\beta)$ a.s.
\item[(b)] If $d < 4/(2-\beta)$, then there exists a positive finite constant $C$ such that 
for any $z \in \R^d$,
\[ \P\Big\{ C L(z, J) \le \mathscr{H}_\varphi(v^{-1}(z) \cap J) < \infty \text{ for all }
J \in \mathscr{I}(T) \Big\} = 1. \]
In particular, $\dim [v^{-1}(z) \cap T] = 2-d(2-\beta)/2$ a.s.\ 
on the event $\{ L(z, T) > 0 \}$.\\
If $d \ge 4/(2-\beta)$, then $v^{-1}(z) \cap T = \varnothing$ a.s.
\end{enumerate}
\end{theorem}

\begin{proof}
Note that it suffices to prove the theorem for all sufficiently small rectangles in $T$ 
because we can cover $T$ with finitely many small rectangles.
Then, without loss of generality, we can assume that the sides of $T$ have length 
$\le \delta_0/\sqrt 2$, so that $v$ satisfies \eqref{SWELND} on $T$.
Together with \eqref{SWE1}, we see that $v$ satisfies Assumption \ref{a1}.
Moreover, Lemmas \ref{SWEa2} and \ref{SWEa3} show that 
$v$ satisfies Assumptions \ref{a2} and \ref{a3} respectively.
Hence, the results follow from Theorems \ref{thm1} and \ref{thm2}.
\end{proof}

Finally, by Theorem \ref{thm:LIL}, 
we get a Chung-type LIL for the solution of \eqref{SWE}.

\begin{theorem}
Let $T$ be a compact rectangle in $(0, \infty) \times \R$. Then for any fixed 
$(t_0, x_0) \in T$, there exists a constant $K$ depending on $(t_0, x_0)$ such that 
\[ \liminf_{r \to 0} \sup_{(t, x): |t-t_0| \vee |x-x_0|\le r} 
\frac{|U(t, x) - U(t_0, x_0)|}{r^{(2-\beta)/2}(\log\log(1/r))^{-(2-\beta)/2}} = K 
\quad \text{a.s.} \]
and $K_1 \le K \le K_2$, where $K_1$ and $K_2$ are positive finite constants 
that depend on $T$.
\end{theorem}

\bigskip

{\bf Acknowledgements.}
The author wishes to thank Thomas Mountford and Yimin Xiao for stimulating discussions
and helpful comments that have led to improvements in the paper.

\end{document}